\documentclass[a4paper,12pt]{amsart}
\usepackage[all]{xy}
\usepackage{wrapfig}
\usepackage[dvips]{graphicx}
\usepackage{amsmath,amsthm,amssymb}

\theoremstyle{definition}
\newtheorem{thm}{Theorem}[section]
\theoremstyle{definition}
\newtheorem{cor}[thm]{Corollary}
\theoremstyle{definition}
\newtheorem{lem}[thm]{Lemma}
\theoremstyle{definition}
\newtheorem{df}[thm]{Definition}
\theoremstyle{definition}
\newtheorem{prop}[thm]{Proposition}
\theoremstyle{definition}
\newtheorem{ex}[thm]{Example}
\theoremstyle{definition}
\newtheorem{rem}[thm]{Remark}
\theoremstyle{definition}

\theoremstyle{definition}

\usepackage{amscd}

\makeindex
\title{Cuntz-Krieger type uniqueness theorem for topological higher-rank graph $C^*$-algebras}
\author{Shinji Yamashita}
\address{Graduate School of Mathematics, Kyushu University, Fukuoka 819-0395, JAPAN}
\email{s-yamashita@math.kyushu-u.ac.jp}
\date{\today}

\begin{document}
\maketitle
%\tableofcontents
%\listoffigures
%\listoftables

\begin{abstract}
We study Sims-Yeend's product system $C^*$-algebras and topological higher-rank graph $C^*$-algebras by Yeend. 
We give a relation between Katsura's Cuntz-Pimsner covariance and Sims-Yeend's one by a direct approach and an explicit form of the core of product system $C^*$-algebras.
Finally, we prove Cuntz-Krieger type uniqueness theorem for topological higher-rank graph $C^*$-algebras under a certain aperiodic condition.
\end{abstract}

%%%%%%%%%%%%%%%%%%%%%%%%%%%%%%%%%%%%%%%%%%%%%%%%%%%%%%%%%%%%%%%
%Introduction
%%%%%%%%%%%%%%%%%%%%%%%%%%%%%%%%%%%%%%%%%%%%%%%%%%%%%%%%%%%%%%%

\section{Introduction}
\label{sec:Introduction}

In 1980, Cuntz and Krieger introduced a class of $C^*$-algebras associated with Markov shifts. These $C^*$-algebras of this class are generated by a family of non-zero partial isometries with so-called Cuntz-Krieger relations. 
They showed if a 0-1 matrix which gives a Markov shift satisfies condition (I), then a $C^*$-algebra generated by a family of any non-zero isometries with the Cuntz-Krieger relation which comes from this Markov shift is isomorphic to the universal one.
It is now called \textit{Cuntz-Krieger uniqueness theorem} and this theorem plays a key role in analyzing structure of Cuntz-Krieger algebras. 
In the proof of this theorem, we need a deep understanding of underlying dynamical systems and covariant relations. 
Our main purpose in this paper is to show a kind of this theorem for topological higher-rank graph $C^*$-algebras.

Since Cuntz and Krieger introduced $C^*$-algebras mentioned above, many authors considered various constructions motivated by Cuntz-Krieger algebras. 
In \cite{Pimsner}, Pimsner provided a class of $C^*$-algebras arising from full Hilbert $C^*$-bimodules with injective left actions, and these $C^*$-algebras are now called Cuntz-Pimsner algebras. 
After Pimsner's work, several authors have challenged to remove some technical assumptions in the construction of Cuntz-Pimsner algebras and tried to unify some constructions of $C^*$-algebras. 
Along this line, Katsura introduced the notions of topological graphs and topological graph $C^*$-algebras generalizing both graph $C^*$-algebras and homeomorphism $C^*$-algebras \cite{Katsura1}. 
In his subsequent papers, he studied relations between dynamical systems of topological graphs and properties of topological graph $C^*$-algebras. 
Furthermore he examined that a class of topological graph $C^*$-algebras includes many $C^*$-algebras which were previously given by various ways. 
In particular, in \cite{Katsura1}, Katsura showed Cuntz-Krieger type uniqueness theorem under an aperiodic condition which is called topological freeness. 
At the same time, for a general Hilbert $A$-bimodule (it is also called a $C^*$-correspondence), Katsura has proposed the appropriate analogue of the Cuntz-Pimsner algebra \cite{Katsura2}. For a full Hilbert $A$-bimodule with an injective left action, Katsura's $C^*$-algebra coincides with Pimsner's one. Katsura also  investigated fundamental properties \cite{Katsura2} and ideal structures \cite{Katsura3}. 

On the other hand, graph $C^*$-algebras were extended to higher dimension. Higher-rank graphs and higher-rank graph $C^*$-algebras were defined by Kumjian and Pask \cite{Kumjian-Pask}. Raeburn, Sims and Yeend removed some technical assumptions of the higher-rank graphs and they called finitely aligned higher-rank graphs. 
Moreover they defined finitely aligned higher-rank graph $C^*$-algebras 
and showed Cuntz-Krieger type uniqueness theorem for finitely aligned higher-rank graph $C^*$-algebras \cite{Raeburn-Sims-Yeend1}, \cite{Raeburn-Sims-Yeend2}. 
Using this theorem, Sims analyzed structure of gauge-invariant ideals of finitely aligned higher-rank graph $C^*$-algebras \cite{Sims1}. 
Subsequently, as a unification of the topological graphs and the finitely aligned higher-rank graphs, Yeend introduced the notion of topological higher-rank graphs and using the groupoid construction \cite{Renault}, he defined topological higher-rank graph $C^*$-algebras \cite{Yeend1}, \cite{Yeend2}.

As soon as Cuntz-Pimsner algebras were introduced, Fowler considered $C^*$-algebras associated with product systems of Hilbert $C^*$-bimodules over  quasi-lattice ordered groups \cite{Fowler}.
Unfortunately the Cuntz-Pimsner covariance of Fowler's $C^*$-algebras does not match with the corresponding covariances of Katsura's $C^*$-correspondences and finitely aligned higher-rank graph $C^*$-algebras. 
However, Sims-Yeend recently proposed a new Cuntz-Pimsner covariance which succeeded to overcome this problem \cite{Sims-Yeend}. 
In \cite{Carlsen-Larsen-Sims-Vittadello}, the authors showed topological higher-rank graph $C^*$-algebras are able to be defined by using the construction of the product system $C^*$-algebras or Cuntz-Nica-Pimsner algebras in the sense of Sims-Yeend and a $C^*$-algebra by this construction is isomorphic to Yeend's one.

As a remarkable result, Sims-Yeend showed the equivalence of Pimsner-Katsura's covariance and that of Sims-Yeend's one by using gauge-invariant uniqueness theorem. 
However there is a huge difference apparently between the notions of two covariances.
In this paper, we pay attention to Pimsner's work in \cite{Pimsner} and point out that there is a relation with Sims-Yeend's one.
Then we establish a connection of two covariances by an analog with Pimsner's work.
Moreover our treatment has an advantage that we can provide a model of the core of Cuntz-Pimsner algebras and Cuntz-Nica-Pimsner algebras. 
This is also applied to prove Cuntz-Krieger uniqueness theorem for topological higher-rank graph $C^*$-algebras.

 This paper is organized as follows. In Section \ref{sec:Preliminaries}, first we recall the definitions of Katsura's Cuntz-Pimsner algebras and Sims-Yeend's product system $C^*$-algebras for the lattice ordered group $\mathbb{N}^{\oplus k}$. 
We also introduce a motivated theorem of this paper in Theorem \ref{thm:Pimsneranalysiscore} due to Pimsner.
In Section \ref{sec:CPcov}, we shall show the equivalence of Katsura's Cuntz-Pimsner covariance defined in \cite[Definition 3.4]{Katsura2} and Sims-Yeend's one over $\mathbb{N}$ and an explicit form of the core of Cuntz-Pimsner algebras as a generalization of a part of Theorem \ref{thm:Pimsneranalysiscore}.
 In Section \ref{sec:core}, we describe the core of product system $C^*$-algebras over $\mathbb{N}^{\oplus k}$. This is a higher-rank version of a part of Theorem \ref{thm:Pimsneranalysiscore}.
 In Section \ref{sec:CK-Unique}, we recall the definition of topological higher rank graphs and associated $C^*$-algebras, and we give the notion of aperiodic condition in Definition \ref{df:cond(A)}. Finally we show Cuntz-Krieger type uniqueness theorem under this aperiodic condition.
%%%%%%%%%%%%%%%%%%%%%%%%%%%%%%%%%%%%%%%%%%%%%%%%%%%%%%%%%%%%%%%
%Preliminaries
%%%%%%%%%%%%%%%%%%%%%%%%%%%%%%%%%%%%%%%%%%%%%%%%%%%%%%%%%%%%%%%
\section{Notation and background}
\label{sec:Preliminaries}

In this section, we recall Katsura's Cuntz-Pimsner algebras and Sims-Yeend's product system $C^*$-algebras.

\subsection{Lattice ordered group $\mathbb{N}^{\oplus k}$} 

 We denote the set of natural numbers by $\mathbb{N}=\{ 0,1,2,\cdots \}$ and the integers by $\mathbb{Z}$. 
We denote by $\mathbb{T}$ the group consisting of complex numbers whose absolute values are 1. Given a (semi)group $P$ with identity and $1 \le k \le \infty$, we denote the direct sum of $P$ by $P^{\oplus k}=\oplus_{i=1}^k P$, the direct product of $P$ by $P^k=\Pi_{i=1}^k P$, which have natural (semi)group structure. 
We consider $\mathbb{N}^{\oplus k}$ as an additive semigroup with identity 0.
For $1 \le k \le \infty$, $e_1,\cdots ,e_k$ are standard generators of $\mathbb{N}^{\oplus k}$. 
We write the $i$-th coordinate of $m \in \mathbb{N}^{\oplus k}$ by $m_{(i)}$. 
For $m,n \in \mathbb{N}^{\oplus k}$, we say $m \le n$ if $n_{(i)}-m_{(i)}$ is a non-negative number for all $1 \le i \le k$. Otherwise we write $m \not\le n$. 
We say  $m<n$ if $m \le n$ and $m \neq n$. 
Then $\mathbb{N}^{\oplus k}$ is a lattice ordered group with a least upper bound $m \vee n$ by $(m \vee n)_{(i)}=\max \{ m_{(i)}, n_{(i)} \}$ for $1 \le i \le k$.
For $m,n \in \mathbb{N}^{\oplus k}$, define $m \wedge n \in \mathbb{N}^{\oplus k}$ by $(m \wedge n)_{(i)}=\min \{ m_{(i)}, n_{(i)} \}$. 
 
\subsection{Hilbert $C^*$-bimodules} 
\label{subsec:Hilbert}

Next we shall recall the notion of Hilbert $A$-bimodules.
Let $A$ be a $C^*$-algebra and $Y$ be a right-Hilbert $A$-module with a right $A$-action $Y \times A \ni (\xi ,a) \longmapsto \xi a \in Y$ and 
an $A$-valued inner product $\langle \cdot , \cdot \rangle_Y$. 
We denote by $\mathcal{L}(Y)$ the $C^*$-algebra of the adjointable operators on $Y$.
Given $\xi ,\eta \in Y$, the rank-one operator $\theta_{\xi ,\eta } \in \mathcal{L}(Y)$ is defined by $\theta_{\xi ,\eta }(\zeta)=\xi \langle \eta ,\zeta \rangle_Y$ for $\zeta \in Y$.
The closure of the linear span of rank-one operators is denoted by $\mathcal{K}(Y)$.
We say that $Y$ is a \textit{Hilbert $A$-bimodule} if $Y$ is a right-Hilbert $A$-module with a homomorphism $\phi : A \longrightarrow \mathcal{L}(Y)$.
Then we can define a left $A$-action on $Y$ by $a \xi=\phi (a)\xi$ for $a \in A$ and $\xi \in Y$. 
We can see $A$ itself is a Hilbert $A$-bimodule with the multiplication from both sides. 
If we emphasize this bimodule structure, we denote it by ${}_AA_A$. 
Let $L :A \ni a \longmapsto L_a \in \mathcal{K}({}_AA_A)$ be a left multiplication. 
Remark that $\mathcal{L}({}_AA_A)$ is the multiplier $C^*$-algebra of $A$ and $\mathcal{K}({}_AA_A)=A$. 
For $S \in \mathcal{K}({}_AA_A)$, there is a unique $a \in A$ so that $S=L_a$.

Let $A,B$ are $C^*$-algebras. Let $Y$ be a right-Hilbert $A$-module and $Y'$ be a right-Hilbert $B$-module. Let $\phi : A \longrightarrow \mathcal{L}(Y')$ be a left action of $Y'$. 
A balanced tensor product $Y \otimes_\phi Y'$ is a vector space spanned by $\xi \otimes \eta$ subject to $\xi a \otimes \eta = \xi \otimes \phi (a) \eta $ for $a \in A$, $\xi \in Y$, $\eta \in Y'$.
$Y \otimes_\phi Y'$ is also a right-Hilbert $B$-module: a right $B$-action $(\xi \otimes \eta )b = \xi \otimes (\eta b)$ and a $B$-valued inner product $\langle \xi_1 \otimes \eta_1 , \xi_2 \otimes \eta_2 \rangle_{Y \otimes_\phi Y'} =\langle \eta_1 , \phi (\langle \xi_1,\xi_2 \rangle_{Y'} )\eta_2 \rangle_Y$.
Similarly, if $Y$ is a Hilbert $A$-bimodule, then $Y \otimes_\phi Y'$ has a left $A$-action. 
In particular, $Y$ and $Y'$ are Hilbert $A$-bimodules, then $Y \otimes_\phi Y'$ is a Hilbert $A$-bimodule.
For $S \in \mathcal{L}(Y)$, we shall define $S \otimes_{\phi} 1_{Y'} \in \mathcal{L}(Y \otimes_\phi Y')$ so that $(S \otimes_\phi 1_{Y'})(\xi \otimes \eta )=S\xi \otimes \eta$. 
If we do not need the form of a left action $\phi$, then we write $Y \otimes_A Y'$, $S \otimes_A 1$ for $Y \otimes_\phi Y'$, $S \otimes_\phi 1$.

\subsection{Cuntz-Pimsner covariances and Cuntz-Pimsner algebras} 
\label{subsec:CP-alg}

We recall the definition of Cuntz-Pimsner algebras. We follow Katsura's definition. 
For a Hilbert $A$-bimodule $X$ and $n \in \mathbb{N}$ with a left action $\phi : A \longrightarrow \mathcal{L}(X)$, we define a right-Hilbert $A$-module $X^{\otimes n}$ by $X^{\otimes 0}=A$, $X^{\otimes 1}=X$, and $X^{\otimes (n+1)} =X \otimes_A X^{\otimes n}$ for $n \ge 1$. 
We denote the $A$-valued inner product of $X^{\otimes n}$ by $\langle \cdot , \cdot \rangle_A^n$. 
Then $X^{\otimes n}$ is a Hilbert $A$-bimodule with a left action $\phi_n : A \longrightarrow \mathcal{L}(X^{\otimes n})$ where $\phi_0(a)=L_a$, $\phi_1(a)=\phi (a)$ and $\phi_{n+1}(a)=\phi (a) \otimes 1_{X^{\otimes n}}$ for $n \ge 1$ and $a \in A$. 
Let $X$ be a Hilbert $A$-bimodule with a left action $\phi$ and $B$ be a $C^*$-algebra. A pair consisting of a *-homomorphism $\pi : A \longrightarrow B$ and a linear map $t : X \longrightarrow B$ is a \textit{representation} of $X$ to $B$ if it satisfies $t(\xi )^*t(\eta )=\pi (\langle \xi , \eta \rangle )$ for $\xi , \eta \in X$ and $\pi (a)t(\xi )=t(\phi (a)\xi )$ for $a \in A , \eta \in X$. 
We denote by $C^*(\pi ,t)$ the $C^*$-algebra generated by the image of $\pi$ and $t$ in $B$.  
For a representation $(\pi ,t)$, we set $t_0=\pi$, $t_1=t$, and for $n=2,3,\cdots $, we define a linear map $t_n : X^{\otimes n} \longrightarrow B$ by $t_n(\xi \otimes \eta )=t(\xi )t_{n-1}(\eta )$ for $\xi \in X$, $\eta \in X^{\otimes n}$. 
Then $(\pi ,t_n)$ is a representation of $X^{\otimes n}$ and 
\begin{equation*}
C^*(\pi ,t)=\overline{\mathrm{span}} \{ t_n (\xi )t_m(\eta )^* \ |\ \xi \in X^{\otimes n},\ \eta \in X^{\otimes m},\ n,m \in \mathbb{N} \} .
\end{equation*} 
For a representation $(\pi ,t)$ of a Hilbert $A$-bimodule $X$, we define a *-homomorphism $t^{(n)} : \mathcal{K}(X^{\otimes n}) \longrightarrow C^*(\pi ,t)$ by $t^{(n)} (\theta_{\xi , \eta })=t_n(\xi ) t_n(\eta )^*$.
Define an ideal $J_X$ of $A$ by $J_X = \phi^{-1}(\mathcal{K}(X)) \cap (\ker \phi )^\perp = \{ a \in A | \phi (a) \in \mathcal{K}(X) \mathrm{\ and \ }ab=0 \mathrm{\ for \ all \ } b \in \ker \phi \} $. 
A representation $(\pi ,t)$ is said to be \textit{Cuntz-Pimsner covariant} if we have $\pi (a)=t^{(1)}(\phi (a))$ for $a \in J_X$. We denote the universal Cuntz-Pimsner covariant representation by $(\pi_X ,t_X)$. We define the $C^*$-algebra $\mathcal{O}_X$ by $\mathcal{O}_X=C^*(\pi_X,t_X)$.

For a representation $(\pi ,t)$, define $C^*(\pi ,t)^{\mathrm{core}}=\overline{\mathrm{span}} \{ t_n(\xi )t_n(\eta )^*\ | \ n \in \mathbb{N},\ \xi,\eta \in X^{\otimes n} \}$. 
The $C^*$-subalgebra $C^*(\pi ,t)^{\mathrm{core}}$ of $C^*(\pi ,t)$ is called the \textit{core} of $C^*(\pi ,t)$.
For $n \in \mathbb{N}$, we define a $C^*$-subalgebra $C_n^\prime$ of $\mathcal{L}(X^{\otimes n})$ by $C_n^\prime=\mathrm{span} \{ S_p \otimes_A 1_{n-p} \ | \ S_p \in \mathcal{K}(X_p),\ 0 \le p \le n \}$ where $1_q=1_{X^{\otimes q}}$ for $q \in \mathbb{N}$. For $n \le m$, we define a map $j^\prime_{m,n} : \mathcal{K}(X^{\otimes n}) \longrightarrow \mathcal{L}(X^{\otimes m})$ by $j^\prime_{m,n}(x)=x \otimes_A 1_{m-n}$. 
If we assume $\phi$ is injective, so is $j_{m,n}^\prime$.
For each $n \in \mathbb{N}$, we define a correspondence $\kappa_n^{\prime (\pi ,t)}: C_n^\prime \longrightarrow C^*(\pi ,t)^{\mathrm{core}}$ by $\kappa _n^{\prime(\pi ,t)} (\sum_{p \le n} S_p \otimes_A 1_{n-p})=\sum_{p \le n} t^{(p)}(S_p)$. 
We do not know whether $\kappa_n^{\prime(\pi ,t)}$ is well-defined or not \textit{a priori}. 
But in \cite[Proposition 3.11]{Pimsner}, Pimsner showed that $\kappa_n^{\prime (\pi ,t)}$ is well-defined for each $n \in \mathbb{N}$ if and only if $(\pi ,t)$ is a Cuntz-Pimsner covariant representation (see also \cite[Corollary 4.8]{Fowler-Muhly-Raeburn}).

\begin{thm}[Pimsner]
\label{thm:Pimsneranalysiscore}
\normalfont\slshape
Let $X$ be a Hilbert $A$-bimodule with an injective left action $\phi$ and $(\pi ,t)$ be a representation. Then 
$(\pi ,t)$ is a Cuntz-Pimsner covariant representation if and only if for each $n \in \mathbb{N}$, $\kappa_n^{\prime (\pi ,t)}$ is a well-defined *-homomorphism. 
Moreover if $\pi$ is injective, then so is $\kappa_n^{\prime (\pi ,t)}$ and 
\begin{equation*}
\kappa^{\prime (\pi ,t)}:=\varinjlim_{n \in \mathbb{N}} \kappa_n^{\prime (\pi ,t)} : \varinjlim_{n \in \mathbb{N}} C_n^\prime \longrightarrow C^*(\pi ,t)^{\mathrm{core}}
\end{equation*}
is isometry.
\end{thm}

The aim of Section \ref{sec:CPcov} and Section \ref{sec:core} in this paper is to show that Sims-Yeend's definition of covariances of a product system $C^*$-algebra can be seen a variant of the well-definedness of $\kappa_n^{(\pi ,t)}$ and to establish an extension of Theorem \ref{thm:Pimsneranalysiscore}.

\subsection{Cuntz-Nica-Pimsner covariances and Cuntz-Nica-Pimsner $C^*$-algebras} 
\label{subsec:CNP-alg}

In this paper we treat product system $C^*$-algebras over $\mathbb{N}^{\oplus k}$, so we consult \cite{Sims-Yeend} for product system $C^*$-algebras over general quasi-lattice ordered groups. 
Let $A$ be a $C^*$-algebra. 
A \textit{product system} $X$ over $\mathbb{N}^{\oplus k}$ of Hilbert $A$-bimodules is a semigroup with an operation $X \times X \ni (\xi , \eta ) \longmapsto \xi \eta \in X$ fibred by $\mathbb{N}^{\oplus k}$ so that each fibre $X_n$ is a Hilbert $A$-bimodule with an $A$-valued inner product $\langle \cdot , \cdot \rangle_A^n$ and a left action $\phi_n$ such that (1) the identity fibre $X_0$ is equal to the bimodule ${}_AA_A$; (2) for each $m,n \in \mathbb{N}^{\oplus k} \setminus \{ 0 \}$, the bilinear map $M_{m,n} : X_m \otimes_A X_n \longrightarrow X_{m+n}$ so that $M_{m,n}(\xi \otimes \eta )=\xi \eta$  is isomorphic ; (3) for $a \in A$, $\xi \in X_n$, $a\xi$ and $\xi a$ coincide with the left action and  right action on $X_n$ respectively. 

 Let $B$ be a $C^*$-algebra and $X$ be a product system over $\mathbb{N}^{\oplus k}$. 
 Let $\psi : X \longrightarrow B$ is a map and $\psi_n=\psi |_{X_n}$ for $n \in \mathbb{N}^{\oplus k}$. We say $\psi$ is a \textit{representation} of $X$ to $B$ if (1) for each $n \in \mathbb{N}^{\oplus k}$, $(\psi_0, \psi_n)$ is a representation of a Hilbert $A$-bimodule $X_n$; (2) $\psi_m(\xi)\psi_n(\eta)=\psi_{m+n}(\xi \eta )$ for $m,n \in \mathbb{N}^{\oplus k}$ and $\xi \in X_m$, $\eta \in X_n$. 
 Define $\psi^{(n)}: \mathcal{K}(X_n) \longrightarrow B$ by $\psi^{(n)} (\theta_{\xi ,\eta } )=\psi_n(\xi )\psi_n(\eta )^*$.

Next, let us recall the notion of a Nica covariance. 
For $S \in \mathcal{K}(X_n)$ and $n \le m \in \mathbb{N}^{\oplus k}$, define a *-homomorphism $\iota_n^m : \mathcal{K}(X_n) \longrightarrow \mathcal{L}(X_m)$ by
\[ 
\iota_n^m(S)=
\left\{
  \begin{array}{ll}
  M_{n,m-n} \circ (S \otimes_A 1_{m-n}) \circ M_{n,m-n}^{-1} & \mbox{if $0 < n < m$} \\
  \phi_m(a) & \mbox{if $n=0$, $S=L_a \in \mathcal{K}(X_0)$ for some $a \in A$}\\
  S & \mbox{if $m=n$}
  \end{array}
\right.
\]
where $1_p:=1_{X_p}$ for $p \in \mathbb{N}^{\oplus k}$. 
 We say that $X$ is \textit{compactly aligned} if for all $m,n \in \mathbb{N}^{\oplus k}$, $S_1 \in \mathcal{K}(X_m)$ and $S_2 \in \mathcal{K}(X_n)$, we have $\iota_m^{m \vee n} (S_1) \iota_n^{m \vee n}(S_2) \in \mathcal{K}(X_{m \vee n})$.
When a product system $X$ is compactly aligned, we say that $\psi$ is a \textit{Nica covariant representation} if for all $m,n \in \mathbb{N}^{\oplus k}$, $S_1 \in \mathcal{K}(X_m)$ and $S_2 \in \mathcal{K}(X_n)$, the equation $\psi^{(m)}(S_1)\psi^{(n)}(S_2)=\psi^{(m \vee n)}(\iota_m^{m \vee n}(S_1) \iota_n^{m \vee n}(S_2))$ holds.
 For a Nica covariance $\psi :X \longrightarrow B$, define $C^*(\psi)=\overline{\mathrm{span}} \{ \psi_m(\xi )\psi_n(\eta )^* | m,n \in \mathbb{N}^{\oplus k} , \xi \in X_m ,\eta \in X_n \}$ and $C^*(\psi )^{\mathrm{core}}=\overline{\mathrm{span}} \{ \psi_n(\xi )\psi_n(\eta )^* | n \in \mathbb{N}^{\oplus k} , \xi ,\eta \in X_n \}$. We call $C^*(\psi )^\mathrm{core}$ the \textit{core} of $C^*(\psi )$. 
By \cite[Proposition 5.10]{Fowler}, if $X$ is compactly aligned, then $C^*(\psi )$ and $C^*(\psi )^{\mathrm{core}}$ are $C^*$-subalgebras of $B$. 
In \cite{Fowler}, it is shown that there exist a $C^*$-algebra $\mathcal{T}_{\mathrm{cov}}(X)$ and a Nica covariant representation $i_X$ of $X$ to $\mathcal{T}_{\mathrm{cov}}(X)$ which is the universal one in the sense that (1) $\mathcal{T}_{\mathrm{cov}}(X)$ is generated by $\{ i_X(\xi ) | \xi \in X \}$; (2) if $\psi$ is any Nica covariant representation of $X$ to a $C^*$-algebra $B$ then there is a unique homomorphism $\psi_* : \mathcal{T}_{\mathrm{cov}}(X) \longrightarrow B$ such that $\psi_* \circ i_X = \psi$.

Let us recall the notion of a Cuntz-Nica-Pimsner covariant representation. 
For a right-Hilbert $A$-module $Y$ and an ideal $I$ of $A$, set $YI=\{ \xi \in Y \ |\ \langle \xi , \xi \rangle_Y \in I \}$. 
For $n \in \mathbb{N}^{\oplus k}$, define an ideal of $A$ by
\begin{eqnarray*} 
I_n=
\left\{
  \begin{array}{ll}
  A & \mbox{if $n=0$} \\
  \bigcap_{0 < m \le n} \ker \phi_m & \mbox{if $n \neq 0$}.
  \end{array}
\right.
\end{eqnarray*} 
For $p,n \in \mathbb{N}^{\oplus k}$ with $p \le n$, define a *-homomorphism $\widetilde{\iota}_p^n : \mathcal{K}(X_p) \longrightarrow \bigoplus_{0 \le q \le n} \mathcal{L}(X_q I_{n-q} )$ by
\begin{equation*}
\widetilde{\iota}_p^n (S)=\bigoplus_{p \le q \le n} \Bigl( \iota_p^q(S)|_{X_qI_{n-q}} \Bigr) \oplus \bigoplus_{0 \le q \le n, p \not\le q} \Bigl( 0|_{X_qI_{n-q}} \Bigr) \in \bigoplus_{0 \le q \le n} \mathcal{L}(X_q I_{n-q} )
\end{equation*}
We say that a statement $P(q)$ indexed by $q \in \mathbb{N}^{\oplus k}$ is true \textit{for large $q$} if for every $n \in \mathbb{N}^{\oplus k}$, there exists $m \in \mathbb{N}^{\oplus k}$ such that $n \le m$ and $P(q)$ holds for $m \le q$. 
We say that a Nica covariant representation $\psi$ is a \textit{
Cuntz-Nica-Pimsner covariant (or CNP-covariant) representation} if 
\begin{align*}
&\sum_{p \in F}\psi^{(p)}(S_p)=0 \mbox{ whenever }F \subset \mathbb{N}^{\oplus k} \mbox{ is finite, } S_p \in \mathcal{K}(X_p) \mbox{ for each }p \in F ,\mbox{ and } \\
&\sum_{p \in F} \widetilde{\iota}_p^q (S_p)=0 \mbox{ for large }q.& 
\end{align*}
Let $i_X$ be the universal Nica covariant representation and denote by $\mathcal{I}$ an ideal of $\mathcal{T}_{\mathrm{cov}}(X)$ generated by 
\begin{align*}
\Bigl\{ \sum_{p \in F}\psi^{(p)}(S_p) \Bigl| & F \subset \mathbb{N}^{\oplus k} \mbox{ is finite, } S_p \in \mathcal{K}(X_p) \mbox{ for each }p \in F ,\mbox{ and } \\
&\sum_{p \in F} \widetilde{\iota}_p^q (S_p)=0 \mbox{ for large }q \Bigr\} .& 
\end{align*}
Define $\mathcal{NO}_X=\mathcal{T}_{\mathrm{cov}}(X) / \mathcal{I}$ and let $q_X$ be the quotient map $q_X : \mathcal{T}_{\mathrm{cov}}(X) \longrightarrow \mathcal{NO}_X$. We call $\mathcal{NO}_X$ the \textit{Cuntz-Nica-Pimsner algebra of $X$}. Let $j_X : X \longrightarrow \mathcal{NO}_X$ be the composition map $j_X=q_X \circ i_X$. Then $j_X$ is the universal CNP-representation in the sense that if $\psi :X \longrightarrow B$ is a CNP-covariant representation of $X$, then there is a unique homomorphism $\Pi \psi : \mathcal{NO}_X \longrightarrow B$ such that $\psi = \Pi \psi \circ j_X$. 
If we assume a product system $X$ is fibred by $\mathbb{N}^{\oplus k}$, then it was shown in \cite{Sims-Yeend} that $j_X$ is injective.
We shall rewrite these definitions in the style of Theorem \ref{thm:Pimsneranalysiscore}. 

For any $n \in \mathbb{N}^{\oplus k}$, set
\begin{equation*}
C_n=\mathrm{span} \{ \widetilde{\iota}_p^n (S_p) \ |\ 0 \le p \le n,\ S_p \in \mathcal{K}(X_p) \} \subset \bigoplus_{0 \le p \le n} \mathcal{L}(X_p I_{n-p} ).
\end{equation*}
Then the multiplication is closed in $C_n$ since $\widetilde{\iota}_p^n (S_p)\widetilde{\iota}_q^n (S_q)=\widetilde{\iota}_{p \vee q}^n (\iota_p^{p \vee q}(S_p)\iota_q^{p \vee q}(S_q))$. 

\begin{prop}
\normalfont\slshape
For each $n \in \mathbb{N}^{\oplus k}$, 
$C_n$ is a $C^*$-subalgebra of $\bigoplus_{0 \le p \le n} \mathcal{L}(X_p I_{n-p} )$.
\end{prop}

\begin{proof}
For a finite subset $F \subset \mathbb{N}^{\oplus k}$, define $B_F= \mathrm{span}(\widetilde{\iota}_p^n(S_p)\ | p \in F,\ S_p \in \mathcal{K}(X_p) \}$. 
For $p,q \in \mathbb{N}^{\oplus k}$ with $p,q\le n$, we define $d(p,q)=\sum_{l=1}^k | p_{(l)}-q_{(l)} |$. 
Set $F_i=\{ p \in \mathbb{N}^{\oplus k}\ |\ p \le n,\ d(n,p)=i \}$ and $F_{\le i}=\{ p \in \mathbb{N}^{\oplus k}\ |\ p \le n,\ d(n,p) \le i \}$. 
Then we have $p \vee q \in F_{\le i}$ for $p \le n$ and $q \in F_{\le i}$. 
Fix a finite set $F \subset F_i$ and set $G=F \cup F_{\le i-1}$. 
First we remark that $B_{F_{\le 0} }=B_{ \{ n \} }=\widetilde{\iota}_n^n(\mathcal{K}(X_n))$ is a $C^*$-subalgebra of $\bigoplus_{0 \le p \le n} \mathcal{L}(X_pI_{n-p})$. 
Suppose $B_G$ is a $C^*$-subalgebra of $\bigoplus_{0 \le p\le n} \mathcal{L}(X_pI_{n-p})$. 
Take $p \in F_i$ with $p \notin G$. 
Then $B_{\{ p\} \cup G} = \widetilde{\iota}_p^n (\mathcal{K}(X_p)) + B_G$ is a $C^*$-subalgebra of $\bigoplus_{0 \le p\le n} \mathcal{L}(X_pI_{n-p})$ because
\begin{equation*}
\widetilde{\iota}_p^n (\mathcal{K}(X_p))B_G \subset B_{F_{\le (i-1)}} \subset B_G , \quad 
B_G \widetilde{\iota}_p^n (\mathcal{K}(X_p)) \subset B_{F_{\le (i-1)}} \subset B_G.
\end{equation*}
Repeating this, we obtain that $C_n=B_{\le n}$ is a $C^*$-subalgebra of $\bigoplus_{0 \le p\le n} \mathcal{L}(X_pI_{n-p})$.
\end{proof}

%%%%%%%%%%%%%%%%%%%%%%%%%%%%%%%%%%%%%%%%%%%%%%%%%%%%%%%%%%%%%%%
%Cuntz-Pimsner covariance
%%%%%%%%%%%%%%%%%%%%%%%%%%%%%%%%%%%%%%%%%%%%%%%%%%%%%%%%%%%%%%%

\section{Katsura's Cuntz-Pimsner covariances and CNP-covariances}
\label{sec:CPcov}

In this section, we consider product systems over $\mathbb{N}$ or Hilbert $C^*$-bimodules and associated $C^*$-algebras. 
In this case, we shall show that a representation is Cuntz-Pimsner covariant if and only if it is CNP-covariant by a direct approach.
This equivalence has been shown in \cite[Proposition 5.3 (2)]{Sims-Yeend} by using gauge-invariant uniqueness theorem. 
However we can show it without using gauge-invariant uniqueness theorem. 
Moreover we give an extension of Theorem \ref{thm:Pimsneranalysiscore}.

We prepare some notations (for detail, see \cite[Section 2]{Fowler-Muhly-Raeburn} and \cite[Section 1]{Katsura3}).
Let $I$ be an ideal of $A$ and $Y$ be a right-Hilbert $A$-module. 
We denote by $Y_I$ the quotient space $Y/YI$. 
The quotient maps $A \longrightarrow A/I$ and $Y \longrightarrow Y_I$ are denoted by the same notation $[ \cdot ]_I$. 
The space $Y_I$ has an $A/I$-valued inner product $\langle \cdot , \cdot \rangle_{Y_I}$ and a right action of $A/I$ so that $\langle [\xi ]_I, [\eta ]_I \rangle_{Y_I} =[ \langle \xi , \eta \rangle_Y ]_I$ and $[\xi ]_I [a]_I=[\xi a]_I$ for $\xi , \eta \in Y$ and $a \in A$. 
Since $S(YI) \subset YI$ for $S \in \mathcal{L}(Y)$, we can define $[\cdot ]_I : \mathcal{L}(Y) \longrightarrow \mathcal{L}(Y_I)$ so that $[S]_I [\xi ]_I = [S\xi ]_I$ for $\xi \in Y$.

\begin{lem}
\label{lem:injlem1}
\normalfont\slshape
Let $A,B$ be $C^*$-algebras. Let $X$ be a right-Hilbert $A$-module, $Y$ be a right-Hilbert $B$-module and $\phi :A \longrightarrow \mathcal{L}(Y)$ be a *-homomorphism. Then we have $\| S \otimes_\phi 1_Y \| = \| [S ]_{\ker \phi } \|$. 
\end{lem}

\begin{proof}
Let $\pi : A \longrightarrow A/\ker \phi$ be the quotient map. 
Let $\overline{\phi} : A /\ker \phi \longrightarrow \mathcal{L}(Y)$ be an injective *-homomorphism such that $\phi = \overline{\phi} \circ \pi$.
We want to define $\Phi : X \otimes_\phi Y \longrightarrow X_{\ker \phi } \otimes_{\overline{\phi}} Y$ so that $\Phi \Bigl(\sum_{i=1}^n \xi_i \otimes \eta_i \Bigr) = \sum_{i=1}^n [ \xi_i ]_{\ker \phi} \otimes \eta_i$. 
We fix elements $\xi_1, \cdots , \xi_n \in X$ and $\eta_1 , \cdots ,\eta_n \in Y$. 
Then
\begin{eqnarray*}
\Bigl\langle \sum_{i=1}^n [\xi_i]_{\ker \phi} \otimes \eta_i , \sum_{j=1}^n [\xi_j]_{\ker \phi} \otimes \eta_j \Bigr\rangle_{X_I \otimes_{\overline{\phi}} Y}
&=&\sum_{i,j} \langle \eta_i , \overline{\phi} ( \langle [\xi_i ]_{\ker \phi }, [\xi_j ]_{\ker \phi } \rangle_{A / \ker \phi } ) \eta_j \rangle_Y \\
&=&\sum_{i,j} \langle \eta_i ,\phi ( \langle \xi_i , \xi_j \rangle_A ) \eta_j \rangle_Y \\
&=&\Bigl\langle \sum_{i=1}^n \xi_i \otimes \eta_i , \sum_{j=1}^n \xi_j \otimes \eta_j \Bigr\rangle_{X \otimes_\phi Y}
\end{eqnarray*}
Hence $\Phi$ can be extended on the whole $X \otimes_\phi Y$ and this induces an isomorphism $X \otimes_\phi Y \cong X_{\ker \phi} \otimes_{\overline{\phi}} Y$. Furthermore we have $\Phi \circ (S \otimes 1_Y) =  ([S]_{\ker \phi } \otimes_{\overline{\phi }} 1_Y) \circ \Phi$. Since $\overline{ \phi }$ is injective, we obtain $\| S \otimes_\phi 1_Y \| = \| [S]_{\ker \phi } \otimes_{\overline{\phi }} 1_Y \| = \| [S]_{\ker \phi } \|$.
\end{proof}

\begin{lem}
\label{lem:injlem2}
\normalfont\slshape
Let $I$ be an ideal of $A$ and $X$ be a right-Hilbert $A$-module. If $S \in \mathcal{L}(X)$ satisfies $[S]_I=0$ and $S|_{XI}=0$, then $S=0$.
\end{lem}

\begin{proof}
$[S]_I=0$ implies $\langle \xi , S \eta \rangle \in I$ for any $\xi ,\eta \in X$. For an approximate unit $\{ h_\lambda \}_\lambda$ of $I$, we have 
$
\langle \xi , S\eta \rangle = \lim_\lambda \langle \xi , S \eta \rangle h_\lambda = \lim_\lambda \langle \xi , S (\eta h_\lambda ) \rangle =0
$
 since we assume $S|_{XI}=0$.
\end{proof}

\begin{prop}
\normalfont\slshape
Let $X$ be a product system over $\mathbb{N}$. 
For $n \le m \in \mathbb{N}$, define $j_{m,n} : C_n \longrightarrow C_m$ by
\begin{equation*}
j_{m,n} \Bigl( \sum_{p \le n} \widetilde{\iota}_p^n (S_p) \Bigr) = \sum_{p \le n} \widetilde{\iota}_p^m (S_p).
\end{equation*}
Then $j_{m,n}$ is a well-defined injective *-homomorphism.
\end{prop}

\begin{proof}
We enough to show for $m=n+1$. For $q \le n$, 
\begin{equation*}
\sum_{p \le n} \iota_p^q(S_p)|_{X_qI_{n-q}}=0 \Longrightarrow \sum_{p \le n} \iota_p^q(S_p)\Bigl|_{X_qI_{n+1-q}} =0
\end{equation*}
since $I_{n+1-q} \subset I_{n-q}$. Moreover
\begin{equation*}
\sum_{p \le n} \iota_p^n(S_p)=0 \Longrightarrow \sum_{p \le n} \iota_p^n (S_p)\Bigl|_{X_nI_1}=0 ,\ \sum_{p \le n} \iota_p^{n+1} (S_p)=0.
\end{equation*}
This implies that $j_{n+1,n}$ is well-defined. 

Next we suppose 
\begin{equation*}
\sum_{p \le n} \widetilde{\iota}_p^{n+1}(S_p) = \sum_{p \le n} \widetilde{j}_{n+1,n} \Bigl( \sum_{p \le n} \widetilde{\iota}_p^n (S_p) \Bigr)=0.
\end{equation*} 
By using Lemma \ref{lem:injlem1}, we have
\begin{equation*}
\Bigl\| \Bigl[ \sum_{p \le n} \iota_p^n(S_p) \Bigr]_{\ker \phi } \Bigr\| = \Bigl\| \Bigl( \sum_{p \le n} \iota_p^n (S_p) \Bigr) \otimes_\phi 1_X \Bigr\| = \Bigl\| \sum_{p \le n} \iota_p^{n+1}(S_p) \Bigr\| =0 .
\end{equation*}
Since $\Bigl[ \sum_{p \le n} \iota_p^n(S_p) \Bigr]_{\ker \phi }=0$ and $\sum_{p \le n} \iota_p^n(S_p)|_{X^{\otimes n}I_1} =0$, we have $\sum_{p \le n} \iota_p^n (S_p)=0$ by using Lemma \ref{lem:injlem2}. Therefore we obtain $\sum_{p \le n} \widetilde{\iota}_p^n (S_p)=0$.
\end{proof}

\begin{rem}
For a product system $X$ over $\mathbb{N}^{\oplus k}$ with $k\ge 2$, $j_{m,n}$ is \textit{not} well-defined in general. 
This happens because, for example, 
\begin{equation*}
\iota_p^{n-e_2}(S_p)|_{X_{n-e_2}I_{e_2}}=0 \Longrightarrow \iota_p^{n-e_2+e_1}(S_p)|_{X_{n-e_2+e_1}I_{e_2}}=0
\end{equation*}
does not hold for $e_2 \le n$, $p \le n$ and $S_p \in \mathcal{K}(X_p)$.
\end{rem}

Let $\psi$ be a representation of a product system $X$ over $\mathbb{N}$. For $n \in \mathbb{N}$, we define a correspondence $\kappa_n^\psi :C_n \longrightarrow C^*(\psi )^{\mathrm{core}}$ by
\begin{equation*}
\kappa_n^\psi \Bigl( \sum_{p \le n} \widetilde{\iota}_p^n (S_p) \Bigr) = \sum_{p \le n} \psi^{(p)}(S_p).
\end{equation*}

We say a representation $(\pi ,t)$ of a Hilbert $A$-bimodule $X$ is CNP-covariant if the representation from $(\pi ,t)$ of the product system over $\mathbb{N}$ defined from $X$ is CNP-covariant.

\begin{prop}
\label{prop:CPequiv1}
\normalfont\slshape
Let $X$ be a Hilbert $A$-bimodule with a left action $\phi$. Then $(\pi ,t)$ is a CNP-covariant representation if and only if for any $n \in \mathbb{N}$, $\kappa_n^{(\pi ,t)}$ is a well-defined *-homomorphism.
\end{prop}

\begin{proof}
Let $(\pi ,t)$ be a CNP-covariant representation. Suppose $\sum_{p \le n} \widetilde{\iota}_p^n (S_p)=0$ where $S_p \in \mathcal{K}(X_p)$ for $0 \le p \le n$. Then for any $m \in \mathbb{N}^{\oplus k}$ with $n \le m$, we have
\begin{equation*}
\sum_{p \le n} \widetilde{\iota}_p^m (S_p)= j_{m,n} \bigl( \sum_{p \le n} \widetilde{\iota}_p^n (S_p) \bigr) =0.
\end{equation*}
Since $(\pi ,t)$ is CNP-covariant, we get $\sum_{p \le n} t^{(p)} (S_p)=0$. Hence $\kappa_n^{(\pi ,t)}$ is well-defined. 

Conversely, we assume that for each $n \in \mathbb{N}$, $\kappa_n^{(\pi ,t)}$ is a well-defined *-homomorphism. 
Then for any finite set $F \subset \mathbb{N}$ and we suppose $\sum_{p \in F} \widetilde{\iota}_p^q (S_p)=0$ for large $q$. Set $n_0=\vee F$. Then there exists $n \in \mathbb{N}$ such that $n_0 \le n$ and $\sum_{p \in F} \widetilde{\iota}_p^n (S_p)=0$. 
If we put $S_p=0$ for $p \notin F$ with $p \le n$, then we have $\sum_{p \le n} \widetilde{\iota}_p^n (S_p)=0$. Since we assume that $\kappa_n^{(\pi ,t)}$ is well-defined, we obtain $\sum_{p \in F} t^{(p)} (S_p)=0$.
\end{proof}

\begin{lem}
\label{lem:cpt1}
\normalfont\slshape
Let $X$ be a Hilbert $A$-bimodule with a left action $\phi$ and $(\pi ,t)$ be a Cuntz-Pimsner representation.  
For any $S \in \mathcal{K}(X^{\otimes n} J_X)$, we have $t^{(n)}(S)=t^{(n+1)}(S \otimes_A 1)$. 
\end{lem}

\begin{proof}
 We enough to show for $S=\theta_{\xi a , \eta}$ where $ \xi ,\eta \in X^{\otimes n}$ and $a \in J_X$.
 First, if we represent $\phi(a)=\sum_{j=1}^\infty \theta_{\xi_i,\eta_i}$ (in the sense of norm convergence) for any $a \in \phi^{-1}(\mathcal{K}(X))$, then we can check
\begin{equation*}
\theta_{\xi a, \eta} \otimes_A 1= \sum_{i=1}^\infty \theta_{\xi \otimes \xi_i, \eta \otimes \eta_i}
\end{equation*}
for any $\xi , \eta \in X^{\otimes n}$. For $a \in J_X$,
\begin{eqnarray*}
t^{(n+1)}(\theta_{ \xi a ,\eta } \otimes_A 1)
&=&\sum_{i=1}^\infty t^{n+1}(\theta_{\xi \otimes \xi_i , \eta \otimes \eta_i})
=\sum_{i=1}^\infty t_n(\xi )t_1 (\xi_i) t_1(\eta_i)^* t_n(\eta )^*\\
&=&t_n(\xi )t^{(1)}(\phi (a)) t_n(\eta )^*
=t_n(\xi )\pi (a) t_n(\eta )^*=t_n(\xi a)t_n(\eta )^*
=t^{(n)}(\theta_{\xi a , \eta }).
\end{eqnarray*}
\end{proof}

\begin{lem}
\label{lem:cpt2}
\normalfont\slshape
Let $X$ be a Hilbert $A$-bimodule with a left action $\phi$ and $n \in \mathbb{N}$. 
Suppose $S \in \mathcal{L}(X^{\otimes n})$ satisfies $S \otimes_A 1 \in \mathcal{K}(X^{\otimes (n+1)})$ and $S|_{X^{\otimes n} \ker \phi}=0$, then $S \in \mathcal{K}(X^{\otimes n} J_X)$. 
\end{lem}

\begin{proof}
We enough to show $\langle \xi , S \eta \rangle_A^n \in J_X$ for $\xi ,\eta \in X^{\otimes n}$. 
From the assumption $S \otimes_A 1 \in \mathcal{K}(X^{\otimes (n+1)})$, we have $\langle \xi , S \eta \rangle_A^n \in \phi^{-1} (\mathcal{K}(X))$ for $\xi ,\eta \in X^{\otimes n}$. 
For any $b \in \ker \phi$,
\begin{equation*}
\langle \xi , S\eta \rangle_A^n b=\langle \xi , S(\eta b) \rangle_A^n =0
\end{equation*}
since $S|_{X^{\otimes n} \ker \phi}=0$. This says that $\langle \xi , S \eta \rangle_A^n \in (\ker \phi )^\perp$. Hence we obtain $\langle \xi , S\eta \rangle_A^n \in J_{X}$.
\end{proof}

\begin{prop}
\label{prop:CPequiv2}
\normalfont\slshape
Let $X$ be a Hilbert $A$-bimodule with a left action $\phi$ and $(\pi ,t)$ be a representation of $X$. 
Then $(\pi ,t)$ is Cuntz-Pimsner covariant if and only if $(\pi ,t)$ is CNP-covariant.
\end{prop}

\begin{proof}
In \cite{Sims-Yeend}, it has shown that if $(\pi ,t)$ is CNP-covariant, then this is Cuntz-Pimsner covariant.
Conversely, we suppose that $(\pi ,t)$ is Cuntz-Pimsner covariant. 
We have to show that $\sum_{p \le n} \widetilde{\iota}_p^n (S_p)=0$ implies $\sum_{p \le n} t^{(p)}(S_p)=0$. 
We can suppose $S_0=L_a$ for some $a \in A$.
The equation $\sum_{p \le n} \widetilde{\iota}_p^n (S_p)=0$ implies
$a \in (\ker \phi )^\perp$ and the following equations:
\begin{equation*}
\Bigl( \phi_m(a)+ \sum_{1 \le p \le m} \iota_p^m (S_p) \Bigr) \Bigl|_{X^{\otimes m} \ker \phi }=0 \ (1 \le m \le n-1) ,\quad 
\phi_n(a)+\sum_{1 \le p \le n} \iota_p^n (S_p)=0.
\end{equation*}
Since
\begin{equation*}
\Bigl( \phi_{n-1}(a) + \sum_{p=1}^{n-1} \iota_p^{n-1}(S_p) \Bigr) \otimes_A 1 =
\phi_n (a) + \sum_{p=1}^{n-1} \iota_p^n(S_p ) = -S_n \in \mathcal{K}(X^{\otimes n})
\end{equation*}
and $\Bigl( \phi_{n-1}(a)+ \sum_{p=1}^{n-1} \iota_p^{n-1}(S_p) \Bigr)\Bigl|_{X^{\otimes (n-1)} \ker \phi }=0$, we get
\begin{equation*}
\phi_{n-1}(a)+ \sum_{p=1}^{n-1} \iota_p^{n-1}(S_p) \in \mathcal{K}(X^{\otimes (n-1)} J_X)
\end{equation*}
by using Lemma \ref{lem:cpt2}. 
Repeating this inductively, for $1 \le m \le n-1$, we obtain
\begin{equation*}
\widetilde{S}_m:=\phi_m(a)+ \sum_{p=1}^m \iota_p^m(S_p) \in \mathcal{K}(X^{\otimes m} J_X)
\end{equation*}
and $a \in J_X$. 
Since we assume that $(\pi ,t)$ is Cuntz-Pimsner covariant, for $\widetilde{S}_1 \in \mathcal{K}(X J_X)$,
\begin{equation*}
\pi(a)+t^{(1)}(S_1)=t^{(1)}(\widetilde{S}_1)=t^{(2)}(\widetilde{S}_1 \otimes_A 1).
\end{equation*}
In the last part, we used Lemma \ref{lem:cpt1}. 
Since $\widetilde{S}_2=\widetilde{S}_1\otimes_A 1 + S_2 \in \mathcal{K}(X^{\otimes 2}J_X)$, we get
$
\pi (a)+t^{(1)}(S_1)+ t^{(2)}(S_2)=t^{(2)}(\widetilde{S}_2)=t^{(3)}(\widetilde{S}_2 \otimes_A 1)
$
by the similar way as above.  
Repeating this, we get 
$
\pi (a)+\sum_{1 \le p \le n-1} t^{(p)}(S_p)=t^{(n)}(\widetilde{S}_{n-1} \otimes_A 1).
$ 
Since $\widetilde{S}_{n-1} \otimes_A 1 + S_n=0$, we conclude
$
\psi_0(a)+\sum_{1 \le p \le n} t^{(p)}(S_p)=0.
$
\end{proof}

From Proposition \ref{prop:CPequiv1} and Proposition \ref{prop:CPequiv2}, we get the following statement.

\begin{thm}
\normalfont\slshape
Let $(\pi ,t)$ be a representation of a Hilbert $A$-bimodule. Then the following statement is equivalent:
\begin{enumerate}
 \item{$(\pi,t)$ is Cuntz-Pimsner covariant;}
 \item{$\kappa_n^{(\pi ,t)}$ is well-defined *-homomorphism for each $n \in \mathbb{N}$;}
 \item{$(\pi ,t)$ is CNP-representation.}
\end{enumerate}
\end{thm}

\begin{thm}
\label{thm:CPanalysiscore}
\normalfont\slshape
Let $X$ be a Hilbert $A$-bimodule and $(\pi, t)$ be a Cuntz-Pimsner covariant. 
Then if $\pi$ is injective, then so is $\kappa_n^{(\pi ,t)}$ for any $n \in \mathbb{N}$. 
Moreover
\begin{equation*}
\kappa^{(\pi ,t)}:=\varinjlim_{n \in \mathbb{N}} \kappa_n^{(\pi ,t)} : \varinjlim_{n \in \mathbb{N}} C_n \longrightarrow C^*(\pi ,t)^{\mathrm{core}}
\end{equation*}
is isometry.
\end{thm}

\begin{proof}
The injectivity will be proved in Proposition \ref{prop:kappa3}. 
We can easily check that $\kappa_m^{(\pi ,t)} \circ j_{m,n}=\kappa_n^{(\pi, t)}$ for $n \le m\in \mathbb{N}$. 
Hence we have done.
\end{proof}

\begin{rem}
We shall note that our explicit form of the core will be very useful to study the properties of the core and Cuntz-Pimsner algebras, relations to dynamical system.
On the other hand, in \cite[Proposition 4.4]{Katsura1} and \cite[Proposition 6.3]{Katsura2}, Katsura showed our Corollary \ref{cor:core} for Cuntz-Pimsner algebras. 
However, Katsura did not give an explicit form of the core like Theorem \ref{thm:Pimsneranalysiscore} or Theorem \ref{thm:CPanalysiscore}. Instead, he used exact sequences of subalgebras of the core and five-lemma.
\end{rem}

%%%%%%%%%%%%%%%%%%%%%%%%%%%%%%%%%%%%%%%%%%%%%%%%%%%%%%%%%%%%%%%
%Analysis of core
%%%%%%%%%%%%%%%%%%%%%%%%%%%%%%%%%%%%%%%%%%%%%%%%%%%%%%%%%%%%%%%
\section{On the core of Cuntz-Nica-Pimsner algebras}
\label{sec:core}

In this section, we give an explicit form of the core of Cuntz-Nica-Pimsner algebras of product systems over $\mathbb{N}^{\oplus k}$. 
This is a higher-rank version of Theorem \ref{thm:Pimsneranalysiscore}.
More precisely, for an injective CNP-covariant representation $\psi$ of a product system $X$ over $\mathbb{N}^{\oplus k}$, we shall describe the core $C^*(\psi )^{\mathrm{core}}$ of $C^*(\psi )$ as an inductive limit $C^*$-algebra without using a representation $\psi$. This means that the structure of the core is independent of the choice of injective CNP-covariant representations.
We note that Lemma \ref{lem:escape} and Proposition \ref{prop:kappa3} were also proved in \cite[Section 3]{Carlsen-Larsen-Sims-Vittadello}, however for the sake of completeness, we give proofs.

\begin{df}
For $n \in \mathbb{N}^{\oplus k}$, we define a $C^*$-algebra by
\begin{equation*}
\widetilde{C}_n =\Bigl\{ \Bigl( \sum_{p \le n} \widetilde{\iota}_p^q (S_p) \Bigr)_{\{ q:n \le q \}} \in \prod_{n \le q} C_q \ \Bigl| \ p \le n,\ S_p \in \mathcal{K}(X_p) \Bigr\} .
\end{equation*}
For $n \le m$, define a *-homomorphism $\widetilde{j}_{m,n} : \widetilde{C}_n \longrightarrow \widetilde{C}_m$ by 
\begin{equation*}
\widetilde{j}_{m,n} \Bigl( \Bigl( \sum_{p \le n} \widetilde{\iota}_p^q (S_p) \Bigr)_{\{ q: n \le q\}} \Bigr)
= \Bigl( \sum_{p \le n} \widetilde{\iota}_p^q (S_p) \Bigr)_{\{ q : m \le q\}}
\end{equation*}
\end{df}
We will show that $\widetilde{j}_{m,n}$ is injective in Corollary \ref{cor:j-inj}.

\begin{df}
Define $\widetilde{\kappa}_n^\psi : \widetilde{C}_n \longrightarrow C^*(\psi )^{\mathrm{core}}$ by
\begin{equation*}
\widetilde{\kappa}_n^\psi : \widetilde{C}_n \ni \Bigl( \sum_{p \le n} \widetilde{\iota}_p^q (S_p) \Bigr)_{\{ q:n \le q \}} \longmapsto \sum_{p \le n} \psi^{(p)}(S_p) \in C^*(\psi )^{\mathrm{core}}.
\end{equation*}
\end{df}
Then $\widetilde{\kappa}_n^\psi$ is well-defined by the notion of the CNP-covariant representation. 
If $\psi$ is injective, then so is $\widetilde{\kappa}_n^\psi$. 

\begin{lem}
\label{lem:escape}
\normalfont\slshape
Let $X$ be a compactly aligned product system of Hilbert $A$-bimodules over $\mathbb{N}^{\oplus k}$, and $\psi$ be a Nica covariant representation of $X$. 
Take $n,m \in \mathbb{N}^{\oplus k}$ such that $m \not\le n$. For $\xi \in X_m$ and $\eta \in X_n I_{(m \vee n)-n}$, we have $\psi_m(\xi )^*\psi_n(\eta )=0$.
\end{lem}

\begin{proof}
For any $\xi_1, \xi_2 \in X_m$, $\eta_1 \in X_nI_{(m \vee n)-n}$, $\eta_2 \in X_n$,
\begin{equation*}
\psi_m(\xi_2)\psi_m(\xi_1)^*\psi_n(\eta_1)\psi_n(\eta_2)^*
=\psi^{(m)} (\theta_{\xi_2 ,\xi_1} ) \psi^{(n)}(\theta_{\eta_1, \eta_2})
=\psi^{(m \vee n)} ( \iota_m^{m \vee n}(\theta_{\xi_2 ,\xi_1}) \iota_n^{m \vee n}(\theta_{\eta_1 ,\eta_2}) ).
\end{equation*}
Since $\eta_1 \in X_n I_{(m \vee n)-n}$ and $(m \vee n)-n \neq 0$, 
$
\iota_n^{m \vee n}(\theta_{\eta_1, \eta_2})(\zeta_1 \zeta_2)
=(\eta_1 \langle \eta_2, \zeta_1 \rangle_A^n ) \zeta_2
$
 holds for $\zeta_1 \in X_n$, $\zeta_2 \in X_{(m \vee n)-n}$. Hence $\iota_n^{m \vee n}(\theta_{\eta_1, \eta_2})=0$. This implies $\psi_m(\xi_2)\psi_m(\xi_1)^*\psi_n(\eta_1)\psi_n(\eta_2)^*=0$. Moreover for any $\xi_3 \in X_m$ and $\eta_3 \in X_n$, 
\begin{equation*}
0=\psi_m(\xi_3)^*\psi_m(\xi_2)\psi_m(\xi_1)^*\psi_n(\eta_1)\psi_n(\eta_2)^*\psi_n(\eta_3)=\psi_m(\xi_1 \langle \xi_2, \xi_3 \rangle_A^m )^*\psi_n(\eta_1 \langle \eta_2 , \eta_3 \rangle_A^n )
\end{equation*}
holds. Since the linear span of elements in the form $\xi_1 \langle \xi_2, \xi_3 \rangle$ is dense in any right-Hilbert module, we finished. 
\end{proof}

\begin{lem}
\label{lem:kappa1}
\normalfont\slshape
Let $X$ be a compactly aligned product system of Hilbert $A$-bimodules over $\mathbb{N}^{\oplus k}$, and $\psi$ be a Nica covariant representation of $X$. 
For $0 \le m \le n,\ 0 \le p \le n$ with $p \not\le m$, $S \in \mathcal{K}(X_p)$, $\xi_1, \xi_2 \in X_mI_{n-m}$, we have
$\psi_m(\xi_1)^*\psi^{(p)}(S)\psi_m(\xi_2)=0$.
\end{lem}

\begin{proof}
This follows from $\psi_m(\xi_1 )^* \psi^{(p)}(\theta_{\eta_1, \eta_2}) \psi_m(\xi_2)=\psi_m( \xi_1)^* \psi_p(\eta_1) \psi_p(\eta_2)^* \psi_m(\xi_2)=0$ $(\eta_1,\eta_2 \in X_p)$ by Lemma \ref{lem:escape}
\end{proof}

\begin{lem}
\label{lem:kappa2}
\normalfont\slshape
Let $X$ be a product system of Hilbert $A$-bimodules over $\mathbb{N}^{\oplus k}$, and $\psi$ be a representation of $X$. 
Suppose $p \le m$ in $\mathbb{N}^{\oplus k}$. Then we have $\psi_m(\xi_1)^* \psi^{(p)}(S)\psi_m(\xi_2)=\psi_0(\langle \xi_1 , \iota_p^m(S) \xi_2 \rangle_A^m )$.
\end{lem}

\begin{proof}
\begin{equation*}
\psi_m(\xi_1)^* \psi^{(p)}(S)\psi_m(\xi_2)=\psi_m(\xi_1)^* \psi_m( \iota_p^m (S) \xi_2)=\psi_0 (\langle \xi_1 , \iota_p^m(S) \xi_2 \rangle_A^m ).
\end{equation*}
\end{proof}

\begin{prop}
\label{prop:kappa3}
\normalfont\slshape
Let $X$ be a compactly aligned product system of Hilbert $A$-bimodules over $\mathbb{N}^{\oplus k}$, and $\psi$ be a Nica covariant representation of $X$. 
Fix $n \in \mathbb{N}^{\oplus k}$.
For $l \le m \in \mathbb{N}^{\oplus k}$ with $n \le m$, $\xi_1, \xi_2 \in X_lI_{m-l}$, we have
\begin{equation*}
\psi_l(\xi_1)^*\Bigl( \sum_{p \le n} \psi^{(p)} (S_p) \Bigr) \psi_l(\xi_2)
= \psi_0\Bigl( \langle \xi_1, \sum_{p\le n,\ p \le l} \iota_p^l (S_p) \xi_2 \rangle_A^m \Bigr)
\end{equation*}
where $S_p \in \mathcal{K}(X_p)$ for each $0 \le p \le n$. 
In particular, if $\psi$ is injective and $\sum_{p \le n} \psi^{(p)}(S_p)=0$, then we have $\sum_{p \le n} \widetilde{\iota}_p^m (S_p)=0$ for all $n \le m$.
\end{prop}

\begin{proof}
By Lemma \ref{lem:kappa1} and \ref{lem:kappa2}, we get
\begin{equation*}
\psi_l(\xi_1)^*\Bigl( \sum_{p \le n} \psi^{(p)} (S_p) \Bigr) \psi_l(\xi_2)
= \psi_l(\xi_1)^*\Bigl( \sum_{p \le m} \psi^{(p)} (S_p) \Bigr) \psi_l(\xi_2)
= \psi_0\Bigl( \langle \xi_1, \sum_{p \le n,p \le l} \iota_p^l (S_p) \xi_2 \rangle_A^l \Bigr)
\end{equation*}
for $l \in \mathbb{N}^{\oplus k}$ with $l \le m$ and $\xi_1, \xi_2 \in X_lI_{m-l}$. If we suppose that $\psi$ is injective and $\sum_{p \le n} \psi^{(p)}(S_p)=0$, then we have $\langle \xi_1, \sum_{p \le n,p\le l} \iota_p^l (S_p) \xi_2 \rangle_A^l =0$. This implies 
$
\sum_{p \le n,p \le l} \iota_p^l (S_p)|_{X_l I_{m-l}}=0.
$
\end{proof}

From the previous proposition, we obtain the following corollary.
\begin{cor}
\label{cor:kappainj}
\normalfont\slshape
Let $X$ be a compactly aligned product system of Hilbert $A$-bimodules over $\mathbb{N}^{\oplus k}$, and $\psi$ be an injective CNP-covariant representation of $X$. 
Then for each $n \in \mathbb{N}^{\oplus k}$ the *-homomorphism $\widetilde{\kappa}_n^\psi :\widetilde{C}_n \longrightarrow C^*(\psi )^{\mathrm{core}}$ is injective. 
\end{cor}

\begin{prop}
\label{prop:j-inj}
\normalfont\slshape
Let $X$ be a compactly aligned product system of Hilbert $A$-bimodules over $\mathbb{N}^{\oplus k}$.
Take $S_p \in \mathcal{K}(X_p)$ for $p \le n$. Then we have 
\begin{equation*}
\sum_{p \le n} \widetilde{\iota}_p^m (S_p) =0 \Longrightarrow \sum_{p \le n} \widetilde{\iota}_p^n (S_p)=0
\end{equation*}
for $m \in \mathbb{N}^{\oplus k}$ with $n \le m$.
\end{prop}

\begin{proof}
We enough to show for $m=n+e_i$. 
Suppose 
\begin{equation*}
\sum_{p \le n} \widetilde{\iota}_p^{n+e_i} (S_p) = 0.
\end{equation*}
Then we have $\sum_{p \le n} \iota_p^n(S_p)|_{X_n\ker \phi_{e_i}}=0$ and $\sum_{p \le n} \iota_p^{n+e_i}(S_p)=0$. 
The latter equation says that $\sum_{p \le n} \iota_p^n(S_p) \otimes_{\phi_{e_i}} 1_{X_{e_i}}=0$. 
By Lemma \ref{lem:injlem1}, we get $\Bigl[ \sum_{p \le n} \iota_p^n (S_p) \Bigr]_{\ker \phi_{e_i}} =0$. By Lemma \ref{lem:injlem2}, we obtain $\sum_{p \le n} \iota_p^n (S_p)=0$. 
Next, for $m < n$, we shall show
\begin{equation*}
\sum_{p \le m} \iota_p^m (S_p) \Bigl|_{X_m I_{n-m}} =0.
\end{equation*}
Thanks to \cite[Lemma 3.15]{Sims-Yeend}, $\widetilde{\phi}_{n-m+e_i}:=\widetilde{\iota}_0^{n-m+e_i} : A \longrightarrow \bigoplus_{0 \le q \le n+m-e_i} \mathcal{L}(X_qI_{n-m+e_i-q})$ is injective. So we shall prove
\begin{equation*}
\Bigl( \sum_{p \le m} \iota_p^m (S_p) \Bigl|_{X_m I_{n-m}} \Bigr) \otimes_{\widetilde{\phi}_{n-m+e_i}} 1_{\bigoplus_{q \le n+m-e_i} X_qI_{n-m+e_i-q}}=0.
\end{equation*}
To prove this, we shall show
\begin{equation*}
\Bigl( \sum_{p \le m} \iota_p^m (S_p) \Bigl|_{X_m I_{n-m}} \Bigr) \otimes_{\phi_q} 1_{X_qI_{(n-m+e_i)-q}}=0.
\end{equation*}
for $q \le n-m+e_i$. 
We shall consider the following three cases for $q$.\\
\textbf{Case.1 $q \neq 0$ and $q \le n-m$: }
Since $X_m I_{n-m} \otimes_{\phi_q} X_qI_{(n-m+e_i)-q}=\{ 0\}$ by the definition of $I_{n-m}$, we have
\begin{equation*}
\Bigl( \sum_{p \le m} \iota_p^m (S_p) \Bigl|_{X_m I_{n-m}} \Bigr) \otimes_{\phi_q} 1_{X_qI_{(n-m+e_i)-q}}=0.
\end{equation*}
\textbf{Case.2 $q=0$: }
Remark that $X_mI_{n-m} \otimes_{\phi_0} I_{n-m+e_i}=X_mI_{n-m+e_i}$ and we now assume 
\begin{equation*}
\sum_{p \le m} \iota_p^m(S_p)\Bigl|_{X_mI_{n-m+e_i}}=0.
\end{equation*}
From these, we obtain
\begin{equation*}
\Bigl( \sum_{p \le m} \iota_p^m (S_p) \Bigl|_{X_m I_{n-m}} \Bigr) \otimes_{\phi_0} 1_{I_{n-m+e_i}}=0.
\end{equation*}
\textbf{Case.3 $q \neq 0$ and $q \not\le n-m$: }
Fix $\xi \in X_mI_{n-m}$ and $\eta \in X_qI_{n-m+e_i}$. Then we have $\xi\eta \in X_{m+q}I_{n-m+e_i}$ and 
\begin{equation*}
\Bigl( \sum_{p \le m} \iota_p^m(S_p)\xi \Bigr) \eta = \sum_{p \le m} \iota_p^{m+q}(S_p)(\xi\eta)=0 
\end{equation*}
by the assumption. Hence we have $\sum_{p \le m} \iota_p^m(S_p)\xi \otimes \eta =0$. Hence we have done for this case.
 
Combining the all cases, we get $\sum_{p \le n} \widetilde{\iota}_p^n(S_p)=0$.
\end{proof}

\begin{cor}
\label{cor:j-inj}
\normalfont\slshape
Let $X$ be a compactly aligned product system of Hilbert $A$-bimodules over $\mathbb{N}^{\oplus k}$.
For $n \le m \in \mathbb{N}^{\oplus k}$, $\widetilde{j}_{m,n}$ is injective. 
\end{cor}

Hence we can consider an inductive limit $C^*$-algebra
\begin{equation*}
\varinjlim_{n \in \mathbb{N}^{\oplus k}} \widetilde{C}_n =\varinjlim_{n \in \mathbb{N}^{\oplus k}} (\widetilde{C}_n, \widetilde{j}_{m,n} ).
\end{equation*}

We have the structure of the core of a $C^*$-algebra $C^*(\psi )$ for a CNP-representation $\psi$ of a compactly aligned product system over $\mathbb{N}^{\oplus k}$. 

\begin{thm}
\label{thm:analysiscore}
\normalfont\slshape
Let $X$ be a compactly aligned product system of Hilbert $A$-bimodules over $\mathbb{N}^{\oplus k}$, and $\psi$ be an injective CNP-covariant representation of $X$. 
Then
\begin{equation*}
\widetilde{\kappa}^\psi:=\varinjlim_{n \in \mathbb{N}^{\oplus k}} \widetilde{\kappa}_n^\psi : \varinjlim_{n \in \mathbb{N}^{\oplus k}} \widetilde{C}_n \longrightarrow C^*(\psi )^{\mathrm{core}}
\end{equation*}
is isometry.
\end{thm}

\begin{proof}
Since we assume $\psi$ is an injective CNP-covariant representation, $\widetilde{\kappa}_n^\psi$ is injective by Corollary \ref{cor:kappainj}. We can check easily $\widetilde{\kappa}_m^\psi \circ \widetilde{j}_{m,n} = \widetilde{\kappa}_n^\psi$ for $n \le m$ in $\mathbb{N}^{\oplus k}$. 
Hence we obtain this theorem.
\end{proof}

Let $X$ be a product system of Hilbert $A$-bimodules over $\mathbb{N}^{\oplus k}$. 
For the universal CNP-covariant representation $j_X : X \longrightarrow \mathcal{NO}_X$, there is an action $\gamma$ of $\mathbb{T}^k=\widehat{\mathbb{Z}^{\oplus k}}$ such that $\gamma_z(j_X(\xi ))=z^n j_X(\xi )$ for $\xi \in X_n$ by the universality. 
Then $\gamma$ is strongly contiunuous and we can define a linear map by
\begin{equation*}
E(x)=\int_{\mathbb{T}^k} \gamma_z (x) dz
\end{equation*}
for $x \in \mathcal{NO}_X$, where $dz$ is the normalized Haar measure. 
Then we have $E(\psi_n (\xi )\psi_m (\eta )^*)=\delta_{n,m} \psi_n (\xi )\psi_m (\eta )^*$ for $\xi \in X_n$, $\eta \in X_m$. Hence $E$ is a faithful conditional expectation onto $C^*(j_X)^{\mathrm{core}}$.
Using this conditional expectation, we can easily check that the core $C^*(j_X)^{\mathrm{core}}$ coincides with the fixed point algebra $\mathcal{NO}_X^\gamma$.

\begin{cor}
\label{cor:core}
\normalfont\slshape
Let $X$ be a compactly aligned product system of Hilbert $A$-bimodules over $\mathbb{N}^{\oplus k}$. 
Let $\psi$ be an injective CNP-covariant representation. 
Then the restriction of the surjection $\Pi \psi :\mathcal{NO}_X \longrightarrow C^*(\psi )$ to the fixed point algebra $\mathcal{NO}_X^\gamma$ is injective.
\end{cor}

\begin{proof}
Since the universal CNP-covariant representation $j_X$ is injective, we get this corollary by Theorem \ref{thm:analysiscore}
\end{proof}

%%%%%%%%%%%%%%%%%%%%%%%%%%%%%%%%%%
% Cuntz-Krieger Uniqueness
%%%%%%%%%%%%%%%%%%%%%%%%%%%%%%%%%%

\section{Cuntz-Krieger type uniqueness theorem for topological higher-rank graph $C^*$-algebras}
\label{sec:CK-Unique}

In this section, we investigate product system $C^*$-algebras associated with topological higher-rank graphs. 
In particular, we shall prove Cuntz-Krieger type uniqueness theorem for compactly aligned topological $k$-graph in the sense of Yeend (\cite{Yeend1}, \cite{Yeend2}) under certain aperiodic condition. We also refer to \cite[Section 5]{Carlsen-Larsen-Sims-Vittadello}. In the case that $\Lambda$ is row-finite without sources, it has done in \cite[Section 3]{Yamashita1}. 

 First, we set up the notations.
 For a locally compact (Hausdorff) space $\Omega$, we denote by $C(\Omega )$ the linear space of all continuous functions on $\Omega$. We define $C_c(\Omega )$, $C_0(\Omega )$, $C_b(\Omega )$ by those of compactly supported functions, functions vanishing at infinity, and bounded functions, respectively.

We say that $(\Lambda ,d)$ is a \textit{topological $k$-graph} if 
(1) a small category $\Lambda$ which has a locally compact topology;   
(2) the range map $r$ and the source map $s$ are continuous and the source map $s$ is locally homeomorphic;  
(3) the composition map of $\Lambda$ is continuous and open; 
(4) the degree map $d : \mathbb{N}^{\oplus k} \longrightarrow \Lambda$ is continuous where we endow $\mathbb{N}^{\oplus k}$ with the discrete topology;  
(5) for all $\lambda \in \Lambda$ and $m, n \in \mathbb{N}^{\oplus k}$ such that $d(\lambda )=m+n$, there exists unique $\mu , \nu \in \Lambda$ such that $d(\mu )=m ,\ d( \nu )=n$, and $\lambda = \mu \nu $. 
By the property (5) of $\Lambda$, for $0 \le n \le m \le l$ in $\mathbb{N}^{\oplus k}$ and $\lambda \in \Lambda^l$, there are unique $\lambda (0,n) \in \Lambda^n$, $\lambda (n,m) \in \Lambda^{m-n}$ and $\lambda (m,l) \in \Lambda^{l-m}$ such that $\lambda = \lambda (0,n) \lambda (n,m) \lambda (m,l)$. Define $\mathrm{Seg}_{(n,m)}^l : \Lambda^l \longrightarrow \Lambda^{m-n}$ by $\mathrm{Seg}_{(n,m)}^l(\lambda )=\lambda (n,m)$.

For $m \in \mathbb{N}^{\oplus k}$, we define $\Lambda^m =d^{-1}(\{ m\} )$ and $r_m=r|_{\Lambda^m}$, $s_m=s|_{\Lambda^m}$. For $U,V \subset \Lambda$, we write $UV=\{ \lambda \mu \ | \ \lambda \in U, \mu \in V , s(\lambda )=r(\mu ) \}$. 
For $U \subset \Lambda^m$, $V \subset \Lambda^n$, define
\begin{equation*}
U \vee V= U\Lambda^{m \vee n -m} \cap V\Lambda^{m \vee n - n}.
\end{equation*}

Next we shall construct a product system over $\mathbb{N}^{\oplus k}$ from a topological $k$-graph $\Lambda$. 
Set $A=C_0(\Lambda^0)$ and for $n \in \mathbb{N}^{\oplus k}$, let $X_n$ be the Hilbert $A$-bimodule associated to the topological graph $(\Lambda^0, \Lambda^n , r|_{\Lambda^n} ,s|_{\Lambda^n})$ in the sense of Katsura \cite[Definition 2.1]{Katsura1}. 
$X_n$ is the completion of the pre Hilbert $A$-bimodule $X_n^{\mathrm{cpt}}:=C_c(\Lambda^n)$ with
\begin{equation*}
\langle \xi ,\eta \rangle^n_A (v)=\sum_{\lambda \in \Lambda^nv} \overline{\xi (\lambda )} \eta (\lambda ),\quad (a \xi b)(\lambda )=a(r(\lambda )) \xi (\lambda ) b( s(\lambda )) 
\end{equation*}
for $n \in \mathbb{N}^{\oplus k}$, $a,b \in A$, $\xi, \eta \in X_n^{\mathrm{cpt}}$ $v \in \Lambda^0$ and $ \lambda \in \Lambda^n$. 
Then $X=\sqcup_{n \in \mathbb{N}^{\oplus k}} X_n$ is a product system over $\mathbb{N}^{\oplus k}$ (see \cite[Proposition 5.9]{Carlsen-Larsen-Sims-Vittadello}).

We say a topological $k$-graph $\Lambda$ is compactly aligned if $U \vee V$ is compact whenever $U$ and $V$ are compact. In \cite[Proposition 5.15]{Carlsen-Larsen-Sims-Vittadello}, it is shown that a topological $k$-graph $\Lambda$ is compactly aligned if and only if the product system arising from $\Lambda$ is compactly aligned. 

Next we define some terms for convenience.
\begin{df}
Given finitely many functions $\xi_1,\cdots ,\xi_L , \eta_1 ,\cdots \eta_L$ of $C_c(\Lambda^m)$, we say $\{ ( \xi_i,\eta_i ) \}_{i=1}^L$ is a \textit{pair of orthogonal functions} for degree $m$ if for any $i=1,\cdots ,L$, $\xi_i(\lambda ) \overline{\eta_i (\lambda ' )}=0$ for $s(\lambda )=s(\lambda ')$ and $\lambda \neq \lambda '$. 
For a set $\Omega \subset \Lambda^m$ and $u_1,\cdots ,u_L \in C_c(\Lambda^m)$, $\{u_i\}_{i=1}^L$ is a \textit{partition of unity} for $\Omega$ if $u_i$ satisfies $0 \le u_i \le 1$, $\sum_{i=1}^L u_i^2(\lambda )=1$ for $\lambda \in \Omega$, and the restriction of $s_m$ to the support $\mathrm{supp}(u_i)$ of $u_i$ is injective. In partcular, $\{ (u_i,u_i) \}_{i=1}^L$ is an pair of orthogonal functions.
\end{df}

\begin{lem}
\normalfont\slshape
\label{lem:unity}
Let $\Lambda$ be a topological $k$-graph and $X$ be the product system arising from $\Lambda$.
\begin{enumerate}
 \item{For $\xi \in C_c(\Lambda^m)$, there exists a partition of unity $\{u_i\}_{i=1}^L$ for the compact support $\mathrm{supp}(\xi )$ of $\xi$.}
 \item{If $\xi \in X_m$ and $\{ u_i\}_{i=1}^L$ is a partition of unity for the support $\mathrm{supp}(\xi )$ of $\xi $, then $\xi=\sum_{i=1}^L u_i \langle u_i , \xi \rangle $.}
\end{enumerate}
\end{lem}

\begin{proof}
(1) Take $\xi \in C_c(\Lambda^m)$. Since $s$ is a local homeomorphism, for each $\lambda \in \Lambda^m$ there exists an relative compact open neighborhood $U_\lambda$ of $\lambda$ such that the restriction of $s_m$ to $U_\lambda$ is injective. Since $\mathrm{supp}(\xi )$ is compact, we can find $\lambda_1 ,\cdots ,\lambda_L \in \Lambda^m$ such that $\mathrm{supp}(\xi ) \subset \cup_{i=1}^L U_{\lambda_i}$. Take functions $v_1,\cdots , v_L$ satisfying $0 \le v_i \le 1$, $\mathrm{supp}(v_i) \subset U_{\lambda_i}$ for each $1 \le i \le L$, and $\sum_{i=1}^L v_i(\lambda )=1$ for all $\lambda \in \mathrm{supp}(\xi )$. Set $u_i:=v_i^{1/2}$. Then $\{ u_i \}_{i=1}^L$ is a partition of unity for $\mathrm{supp}(\xi )$.\\
(2) Take $\xi \in X_m$ and a partition of unity $\{ u_i\}_{i=1}^L$ for $\mathrm{supp}(\xi )$. For $\lambda \in \Lambda^m$,
\begin{equation*}
\Bigl( \sum_{i=1}^L u_i \langle u_i , \xi \rangle \Bigr) (\lambda )
=\sum_{i=1}^L u_i(\lambda ) \Bigl( \sum_{s(\mu )=s(\lambda )}u_i(\mu ) \xi (\mu )  \Bigr)
=\sum_{i=1}^L u_i(\lambda )^2 \xi (\lambda ) = \xi (\lambda ). 
\end{equation*}
\end{proof}

\begin{lem}
\label{lem:turn1}
\normalfont\slshape
Let $\Lambda$ be a topological $k$-graph and $X$ be the product system arising from $\Lambda$, and $\psi$ be a representation of $X$.  
Let $\{u_i\}_{i=1}^L$ be a partition of unity for the support $\mathrm{supp}(\xi)$ of $\xi \in X_n$. Then
\begin{equation*}
\sum_{i=1}^L \psi_n(u_i)x\psi_n(u_i)^*\psi_n(\xi )=\psi_n(\xi ) x
\end{equation*}
for all elements $x$ of the relative commutant algebra $\psi_0(A)' \cap C^*(\psi)$.
\end{lem}

\begin{proof}
For $x \in \psi_0(A)' \cap C^*(\psi)$,
\begin{eqnarray*}
\sum_{i=1}^L \psi_n(u_i)x\psi_n(u_i)^*\psi_n(\xi )
&=&\sum_{i=1}^L \psi_n(u_i)x\psi_0(\langle u_i, \xi \rangle)
=\sum_{i=1}^L \psi_n(u_i)\psi_0(\langle u_i, \xi \rangle)x \\
&=&\psi_n(\sum_{i=1}^L u_i \langle u_i, \xi \rangle )x.
\end{eqnarray*}
Since $\{u_i\}_{i=1}^L$ is a partition of unity for $\mathrm{supp}(\xi)$, we obtain $\sum_{i=1}^L u_i \langle u_i, \xi \rangle =\xi$ by Lemma \ref{lem:unity}. 
\end{proof}

Let $\Lambda$ be a topological $k$-graph. 
For $m \in \mathbb{N}^{\oplus k}$, let us define an injective *-homomorphism $\pi_m :C_b(\Lambda^m ) \longrightarrow \mathcal{L}(X_m)$ by
\begin{equation*}
(\pi_m(Q)\xi )(\lambda )=Q(\lambda ) \xi (\lambda ), \quad Q \in C_b(\Lambda^m),\ \lambda \in \Lambda^m. 
\end{equation*}
The following lemma is proved by Katsura (\cite[Lemma 1.16, 1.17]{Katsura1})

\begin{lem}
\label{lem:katsura}
\normalfont\slshape
Let $\Lambda$ be a topological $k$-graph. 
For each $m \in \mathbb{N}$, the image $\pi_m(C_0(\Lambda^m))$ of $C_0(\Lambda^m)$ is included in the $C^*$-algebra $\mathcal{K}(X_m)$. 
Given $Q \in C_c(\Lambda^m)$, there exists a pair of orthogonal functions $\{(\xi_i, \eta_i )\}_{i=1}^L$ such that
\begin{equation*}
Q=\sum_{i=1}^L \xi_i \overline{\eta_i} ,\quad \mathrm{and} \quad \pi_m(Q)=\sum_{i=1}^L \theta_{\xi_i , \eta_i}.
\end{equation*}
\end{lem}
 
\begin{lem}
\normalfont\slshape
Let $\Lambda$ be a topological $k$-graph and $X$ be the product system arising from $\Lambda$. Let $\psi$ be a representation of $X$.  
Set $\varphi_m^\psi=\psi^{(m)} \circ \pi_m :C_0(\Lambda^m ) \longrightarrow C^*(\psi )$. 
Then, for any $Q \in C_0(\Lambda^m)$, we have $\varphi_m^\psi(Q) \in \psi_0(A)' \cap C^*(\psi )$.
\end{lem}

\begin{proof}
For $a \in A$ and $S \in \mathcal{K}(X_m)$, we can show $\psi_0(a)\psi^{(m)}(S)=\psi^{(m)}(\phi_m(a)S)$ (see \cite[Lemma 2.4]{Katsura2}).
For $Q \in C_0(\Lambda^m)$, we get
\begin{equation*}
\psi_0(a)\varphi_m^\psi(Q)=\psi_0(a) \psi^{(m)}( \pi_m(Q))=\psi^{(m)}(\phi_m(a) \pi_m(Q))=\psi^{(m)}(\pi_m(Q) \phi_m(a))=\varphi_m^\psi(Q) \psi_0(a).
\end{equation*}
\end{proof}

For $Q_1 \in C_b(\Lambda^m),\ Q_2 \in C_b(\Lambda^n)$, define a function $Q_1 \widehat{\otimes} Q_2 \in C_b(\Lambda^{m+n})$ by $(Q_1 \widehat{\otimes} Q_2)(\lambda )=Q_1(\lambda (0,m))Q_2(\lambda (m,m+n))$ for $\lambda \in \Lambda^{m+n}$.

\begin{lem}
\label{lem:turn2}
\normalfont\slshape
Let $\Lambda$ be a topological $k$-graph and $X$ be the product system arising from $\Lambda$. Let $\psi$ be a representation of $X$. 
Let $\{ (u_i,u_i) \}_{i=1}^L$ be a pair of orthogonal functions for degree $m$. Then for $Q \in C_0(\Lambda^n)$, we have
\begin{equation*}
\sum_{i=1}^L \psi_m(u_i)\varphi_n^\psi (Q)\psi_m(u_i)^* =\varphi_{m+n}^\psi (\sum_{i=1}^L |u_i|^2 \widehat{\otimes} Q).
\end{equation*}
\end{lem}

\begin{proof}
It is enough to show for $Q \in C_c(\Lambda^m)$. 
By Lemma \ref{lem:katsura}, there is a pair of orthogonal functions $\{ (\xi_j, \eta_j) \}_{j=1}^M$ such that $Q= \sum_{j=1}^M \xi_j \overline{\eta_j}$. Then $\{ (u_i \otimes \xi_j, u_i \otimes \eta_j) \}_{1 \le i \le L,1 \le j \le M}$ is a pair of orthogonal functions for degree $m+n$ and
\begin{eqnarray*}
\sum_{i=1}^L \psi_m(u_i)\varphi_n^\psi(Q)\psi_m(u_i)^* 
&=&\sum_{i,j}\psi_{m+n}(u_i \otimes \xi_j) \psi_{m+n}(u_i \otimes \eta_j)^* \\
&=& \varphi^\psi_{m+n}(\sum_{i,j} (u_i \widehat{\otimes} \xi_j)\overline{(u_i \widehat{\otimes} \eta_j)} ) \\
&=&\varphi^\psi_{m+n}(\sum_{i=1}^L |u_i|^2 \widehat{\otimes} Q) .
\end{eqnarray*}
\end{proof}

Using the lemmas above, we obtain the following. 
\begin{lem}
\label{lem:turn3}
\normalfont\slshape
Let $\Lambda$ be a topological $k$-graph and $X$ be the product system arising from $\Lambda$. 
Let $\psi$ be a representation of $X$. 
Let $\{u_i\}_{i=1}^L$ be a partition of unity for the support $\mathrm{supp}(\xi)$ of $\xi \in X_m$ and $Q \in C_0(\Lambda^n)$. Then we have
\begin{equation*}
\psi_m(\xi ) \varphi_n^\psi(Q)=\sum_{i=1}^L \varphi_{m+n}^\psi(u_i^2 \widehat{\otimes} Q)\psi_m(\xi )
\end{equation*}
\end{lem}

Next, we introduce aperiodic condition which is used in Cuntz-Krieger type theorem. 
For $\lambda \in \Lambda$ and $p,q \in \mathbb{N}^{\oplus k}$ such that $p \le q$, there exists uniquely $\lambda (p,q) \in \Lambda^{p-q}$ such that $\lambda = \lambda_1 \lambda (p,q) \lambda_2$ where $\lambda_1 \in \Lambda^p$ and $\lambda \in \Lambda^{n-q}$. 
For $m \in (\mathbb{N} \cup \{ \infty \})^k$, define a (discrete) topological $k$-graph $\Omega_{k,m}$ by
\begin{equation*}
\Omega_{k,m}=\{ (p,q) \in \mathbb{N}^{\oplus k} \times \mathbb{N}^{\oplus k} \ | \ p \le q \le m \}, \  r(p,q)=p, \ s(p,q)=q
\end{equation*}
and $d(p,q)=q-p$. We denote $\Omega_{k,m}^0=\{ p \in \mathbb{N}^{\oplus k} \ | \ p \le m \}$.

\begin{df}
Let $\Lambda$ be a topological $k$-graph and $m \in (\mathbb{N} \cup \{ \infty \})^k$. 
We say a morphism $\alpha : \Omega_{k,m} \longrightarrow \Lambda$ is a \textit{boundary path} if $\alpha (n)\Lambda^{e_i}=\emptyset$ for $n \in \mathbb{N}^{\oplus k}$ with $n_{(i)}=m_{(i)}$. 
Let us define $d(\alpha )=m$ for a boundary path $\alpha$ and $\Lambda^{\le \infty}$ be the set of boundary paths. 
For $V \subset \Lambda^0$, we define $V\Lambda^{\le \infty }=\{ \alpha \in \Lambda^{\le \infty } \ | \ \alpha (0) \in V \}$.
\end{df}

\begin{lem}
\normalfont\slshape
For $1 \le k \le \infty$, $v\Lambda^{\le \infty}$ is non-empty.
\end{lem}

\begin{proof}
For the case $1 \le k <\infty$, for any $v \in \Lambda^0$, we can show that $v\Lambda^{\le \infty}$ is not empty by the same way of \cite[Lemma 2.11]{Raeburn-Sims-Yeend2}. 
But we remark that the method of \cite[Lemma 2.11]{Raeburn-Sims-Yeend2} can not be applied for the case  $k=\infty$. 
For $m \le n$, define a surjective map $\pi_{m,n} : v\Lambda^{\le n} \longrightarrow v\Lambda^{\le m}$ by
\[
\pi_{m,n}(\lambda )=
\left\{
  \begin{array}{ll}
  \lambda (0,m) & \mathrm{if\ } m \le d(\lambda ) \\
  \lambda (0,m') & \mathrm{otherwise}
  \end{array}
\right.
\]
where $m'$ is defined by $(m')_{(l)}=m_{(l)}$ if $m_{(l)} \le d(\lambda )_{(l)}$ and otherwise $(m')_{(l)}=d(\lambda )_{(l)}$. We can check $\pi_{m,n}$ is surjective. 
Then 
\begin{equation*}
v\Lambda^{\le \infty} \ni \alpha \longmapsto \{ \alpha (0,l)\}_{l \in \mathbb{N}^{\oplus \infty} , l \le d(\alpha )} \in \varprojlim_{n \in \mathbb{N}^{\oplus \infty}} (v\Lambda^{\le n} , \pi_{m,n} ) 
\end{equation*}
is bijective. By \cite[Proposition 5, p.198]{Bourbaki}, $\varprojlim_{n \in \mathbb{N}^{\oplus \infty}} (v\Lambda^{\le n} , \pi_{m,n} )$ is non-empty.
\end{proof} 

For $p \le \sigma^p\alpha$, define $\sigma^p : \Lambda^{\le \infty} \longrightarrow \Lambda^{\le \infty}$ so that $(\sigma^p\alpha )(m,n) = \alpha (m+p , n+p)$ for $m,n \in \mathbb{N}^{\oplus k}$ with $m \le n \le d(\alpha )-p$ .
We shall define aperiodic condition which is slightly stronger than \cite[Definition 5.1]{Yeend2}.

\begin{df}
\label{df:cond(A)}
We say $\Lambda$ satisfies \textit{aperiodic condition} if for a non-empty open set $V$, there are $v_0 \in V$ and $\alpha \in v_0 \Lambda^{\le \infty}$ such that 
\begin{align*}
p,q \in \mathbb{N}^{\oplus k} \mathrm{\ with}\ p\neq q \le d(\alpha ) \Longrightarrow \sigma^p (\alpha ) \neq \sigma^{q}(\alpha ) \tag{A}
\end{align*}
Remark that if $d(\alpha ) \notin \mathbb{N}^{\oplus k}$ then $\sigma^p (\alpha ) \neq \sigma^q(\alpha )$ means that there is $M=M_{p,q} \le d(\sigma^p (\alpha ) ) \wedge d(\sigma^q (\alpha ) )$ such that $\alpha (p , p +M) \neq \alpha (q ,q +M)$.
\end{df}

First we shall show that aperiodic condition implies with topological freeness which is defined in \cite[Definition 5.4]{Katsura1} when $k=1$.
We recall the definition of topological freeness.

Let $\Lambda$ be a topological 1-graph. 
For $n \ge 1$, $\lambda \in v\Lambda^n v$ is called a \textit{loop} and $v$ is called a \textit{base point} of $\lambda$.
A loop $\lambda \in v\Lambda^nv$ based at $v$ is said to be \textit{without entrances} if for $0 \le m \le n-1$, $\mu \in \Lambda^1$ satisfies $r(\mu )=r(\lambda (m,m+1))$, then $\mu = \lambda (m , m+1)$.

\begin{df}
Let $\Lambda$ be a topological 1-graph. For $n \ge 1$, define
\begin{equation*}
\mathrm{Per}_n:=\{ v \in \Lambda^0 \ |\ \mbox{Every }\lambda \in v\Lambda^n v \mbox{ has no entrance.} \}
\end{equation*}
and $\mathrm{Per}:=\cup_{n=1}^\infty \mathrm{Per}_n$. We say $\Lambda$ is \textit{topologically free} if the set of interior points of $\mathrm{Per}$ is the empty set.
\end{df}

\begin{prop}
\normalfont\slshape
Let $\Lambda$ be a topological 1-graph. Then if $\Lambda $ satisfies aperiodic condition, then $\Lambda$ is topologically free. 
\end{prop}

\begin{proof}
We assume $\Lambda$ satisfies aperiodic condition. 
For any non-empty open set $V$ of $\Lambda^0$, there are $v \in V$ and $\alpha \in v\Lambda^{\le \infty}$ such that $\alpha$ satisfies (A). We shall show that for any $n \ge 1$ and $\lambda \in v\Lambda^n v$, there is an entrance of $\lambda$. 
First we suppose $\alpha = \lambda \alpha$. 
Then we have $\alpha =\lambda \lambda \cdots$. But this contradicts to $\sigma^n\alpha \neq \alpha$. So we get $\alpha \neq \lambda \alpha$. 
This says that there is a minimal $M\ge 0$ such that $\alpha (0,M) \neq (\lambda \alpha )(0,M)$. 
If we suppose $M \le d(\lambda )$, then $\alpha (0, d(\lambda )) \neq \lambda $ and this says the existence of an entrance of $\alpha$. 
Suppose $M>d(\lambda )$. By the minimality of $M$, we have $\alpha (0,d(\lambda ))=\lambda $ and $\alpha (d(\lambda ), M) \neq \lambda (0, M-d(\lambda ))$. Hence we have shown $\lambda$ has an entrance.
\end{proof}
\begin{ex}
\label{ex:BC-alg}
Let $P=\{ p_1,p_2, \cdots | p_1 \le p_2 \le \cdots \}$ be the set of all prime numbers. For each $q \in P$, we define by $\mathbb{Z}_q$ the $q$-adic ring. Define $\mathcal{Z}=\prod_{p \in P} \mathbb{Z}_p$. 
For $n \in \mathbb{N}^{\oplus \infty}$, we set $p^n=( p_i^{n_{(i)}})_{i=1}^\infty$. 
For each $n \in \mathbb{N}^{\oplus \infty}$, define the endomorphism $\alpha_n$ on $C(\mathcal{Z})$ so that
\[
\alpha_{n}(f)(x)=
\left\{
  \begin{array}{ll}
  f(x/p^n) & \mathrm{if}\ x \in p^n \mathcal{Z}=\prod_{i=1}^\infty (p_i^{n_{(i)}}\mathbb{Z}_{p_i})  \\
  0 & \mathrm{otherwise}.
  \end{array}
\right.
\]
Then the Bost-Connes algebra is isomorphic to $C(\mathcal{Z}) \rtimes_\alpha \mathbb{N}^{\oplus \infty}$ which was shown in \cite[Proposition 32]{Laca}. 

For $n \in \mathbb{N}^{\oplus \infty}$, define $\Lambda^0=\mathcal{Z}$, $\Lambda^n= p^n \mathcal{Z}$. 
For $n \in \mathbb{N}^{\oplus \infty}$, $x \in \Lambda^n$, define $r_n(x)=x$. $s_n(x)=x/p^n$. 
From $(\Lambda^0 , \Lambda^n, r_n ,s_n)$, we can construct a topological $\infty$-graph $\Lambda_{\mathrm{BC}}$. 
We remark that for each $n \in \mathbb{N}^{\oplus \infty}$, $r_n(\Lambda^n)$ is not dense in $\Lambda^0$. 
This says that left actions of Hilbert bimodules are not injective. 

Let $X$ be the product system arising from $\Lambda_{\mathrm{BC}}$. Then $\mathcal{NO}_X$ is isomorphic to the Bost-Connes algebra by the universality. 
In particular, if we set $u_n:=\psi_n(1_{p^n\mathcal{Z}})$, we have $u_n u_m=u_{n+m}$ and $u_n \psi_0(f) u_n^*=\psi_0(\alpha_n (f))$.

We want to show that $\Lambda_{\mathrm{BC}}$ satisfies aperiodic condition. 
Take $\alpha \in v\Lambda_{\mathrm{BC}}^{\le \infty}$. If $d(\alpha )_{(l)} = \infty$, then we have $x(0)_l=0$. 
But for each $p \in P$, $\mathbb{Z}_p \setminus \{ 0 \}$ is dense in $\mathbb{Z}_p$. 
Hence we have that $v \in \Lambda_{\mathrm{BC}}^0$ such that every element of $v \Lambda_{\mathrm{BC}}^{\le \infty}$ has a finite degree is dense in $\Lambda_{\mathrm{BC}}^0$. 
This implies that $\Lambda_{\mathrm{BC}}$ satisfies aperiodic condition. 
\end{ex}

We prepare a lemma which is due to Renault-Sims-Yeend. 
We would like to thank Aidan Sims who allows us to use it in here. 

\begin{lem}
\label{lem:R-S-Y}
\normalfont\slshape
Let $\Lambda$ be a compactly aligned topological $k$-graph.
Let $K$ be a compact subset of $\Lambda$ and $\lambda \in \Lambda$. 
If $\{\lambda \} \vee K = \emptyset$ holds, then there is a neighborhood of $\lambda$ such that $V \vee K = \emptyset$.
\end{lem}

\begin{proof}
Let us suppose $V \vee K \neq \emptyset$ for any neighborhood $V$ of $\lambda$. 
Fix a compact neighborhood $V_0$ of $\lambda$ such that $V_0 \vee K \neq \emptyset$. 
For $\lambda \in V \subset V_0$, we can take $\mu_V \in V_0 \vee K$ and $\mu_V(0,d(\lambda ))\in V$. 
Since the net $\{ \mu_V \}_{V \subset V_0 }$ is in the compact set $V_0 \vee K$, we can find a subnet $\{ \mu_{V_j} \}$ of $\{ \mu_V \}$ such that $\mu_{V_j}$ converges to some $\mu \in V_0 \vee K$. 
Since $\mathrm{Seg}_{(0,d(\lambda ))}$ is continuous, we obtain $\mu(0,d(\lambda ))=\lambda$.
Hence we obtain $\mu \in \{ \lambda \} \vee K$, however this contradicts $\{ \lambda \} \vee K = \emptyset$.
\end{proof}

Let $X$ be the product system arising from a compact aligned topological $k$-graph $\Lambda$ and $\psi$ be a CNP-covariant representation of $X$.
Define a linear subspace $C^*(\psi )^{\mathrm{cpt}}$ of $C^*(\psi )$ by 
\begin{equation*}
C^*(\psi )^\mathrm{cpt}=\mathrm{span} \{ \psi_n(\xi ) \psi_m(\eta )^* | n,m \in \mathbb{N}^{\oplus k},\ \xi \in X_n^{\mathrm{cpt}} ,\ \eta \in X_m^{\mathrm{cpt}} \} .
\end{equation*}
By \cite[Lemma 1.6]{Katsura1}, $C^*(\psi )^{\mathrm{cpt}}$ is dense in $C^*(\psi )$.

Next, we shall show the key proposition in this section.

\begin{prop}
\label{prop:freeness}
\normalfont\slshape
Let $\Lambda$ be a compactly aligned topological $k$-graph and $X$ be the product system arising from $\Lambda$. Let $\psi$ be a CNP-covariant representation.
Suppose $\Lambda$ satisfies aperiodic condition and $\psi$ is injective. 
Take $x,x_0 \in C^*(\psi)^{\mathrm{cpt}}$ such that 
\begin{equation*}
x=\sum_{1 \le i \le L} \psi_{n_{i,1}}(\xi_{i,1})\psi_{n_{i,2}}(\xi_{i,2})^*,\quad 
x_0=\sum_{\{ 1 \le i \le L | n_{i,1}=n_{i,2} \}} \psi_{n_{i,1}}(\xi_{i,1}) \psi_{n_{i,2}} (\xi_{i,2})^*
\end{equation*}
where $\xi_{i,j} \in X_{n_{i,j}}^{\mathrm{cpt}}$ $(j=1,2)$.
Then for any $\epsilon >0$ and $x \in C^*(\psi)^{\mathrm{cpt}}$, there exist $b_1,b_2 \in C^*(\psi )$ such that $\| b_1 \| , \| b_2 \| \le 1$, $b_1^*xb_2=b_1^*x_0b_2$ and $\| x_0 \| \le \| b_1^* xb_2 \| + \epsilon$. 
\end{prop}

\begin{proof}
Set $n_0=\bigvee_{i=1}^L (n_{i,1} \vee n_{i,2})$. 
Define $\mathcal{I}_0=\{ 1 \le i \le L \ | \ n_{i,1}=n_{i,2} \}$ and $n_i:=n_{i,1}=n_{i,2}$ for $i \in \mathcal{I}_0$. 
Then
\begin{eqnarray*}
x_0&=&\sum_{i \in \mathcal{I}_0} \psi_{n_i}(\xi_{i,1}) \psi_{n_i} (\xi_{i,2})^*
=\sum_{i \in \mathcal{I}_0} \psi^{(n_i)} (\theta_{\xi_{i,1}, \xi_{i,2}} )\\
&=&\widetilde{\kappa}_{n_0}^\psi \Bigl( \Bigl( \sum_{i \in \mathcal{I}_0}  \widetilde{\iota}_{n_i}^n (\theta_{\xi_{i,1}, \xi_{i,2}}) \Bigr)_{\{ n: n_0 \le n\} } \Bigr).
\end{eqnarray*}
Since we assume $\psi$ is injective, $\widetilde{\kappa}_n^\psi$ is also injective by Corollary \ref{cor:kappainj}. Hence for any $\epsilon >0$, there is some $n \ge n_0$ such that
\begin{equation*}
\| x_0 \| \le \Bigl\| \sum_{i \in \mathcal{I}_0}  \widetilde{\iota}_{n_i}^n (\theta_{\xi_{i,1}, \xi_{i,2}}) \Bigr\|_{\bigoplus_{0 \le p \le n} \mathcal{L}(X_pI_{n-p}) } + \frac{\epsilon}{2}.
\end{equation*}
Then there is $0 \le p \le n$ such that 
\begin{equation*}
\| x_0 \| \le  \Bigl\| \sum_{\{i \in \mathcal{I}_0 |n_i \le p\}}  \iota_{n_i}^p (\theta_{\xi_{i,1}, \xi_{i,2}}) \Bigr|_{X_pI_{n-p}} \Bigr\|_{\mathcal{L}(X_pI_{n-p}) } + \epsilon.
\end{equation*}
Therefore for the above $\epsilon$, there exist some $\eta_1, \eta_2 \in X_p^{\mathrm{cpt}}I_{n-p}$ such that
\begin{eqnarray*}
\| x_0 \| < \Bigl\| \sum_{\{ i \in \mathcal{I}_0 | n_i \le p \}} \langle \eta_1,  \iota_{n_i}^p (\theta_{\xi_{i,1}, \xi_{i,2}}) \eta_2 \rangle_A^p \Bigr\| + \epsilon = \| g_p \| + \epsilon
\end{eqnarray*}
where we put $g_p=\sum_{\{ i \in \mathcal{I}_0| n_i \le p \}} \langle \eta_1,  \iota_{n_i}^p (\theta_{\xi_{i,1}, \xi_{i,2}}) \eta_2 \rangle_A^p \in A$. Then there are $v' \in \Lambda^0$ and non-empty neighborhood $V$ of $v'$ such that $\| x_0 \| < |g_p(v)| +\epsilon $ for any $v \in V$. 
If $0\le p<n$, then by Lemma \ref{lem:escape}, we have
\begin{eqnarray*}
\psi_p(\eta_1)^*x\psi_p(\eta_2)
&=&\sum_{\{ i | n_{i,1},n_{i,2} \le p \} } \psi_p(\eta_1)^* \psi_{n_{i,1}}(\xi_{i,1}) \psi_{n_{i,2}} (\xi_{i,2})^* \psi_p(\eta_2)\\
&=&\sum_{\{ i | n_{i,1},n_{i,2} \le p \} } \psi_{n_{i,1}'} (\xi_{i,1}')^* \psi_{n_{i,2}'}(\xi_{i,2}')
\end{eqnarray*}
where $\xi_{i,j}' \in X_{n_{i,j}'}^{\mathrm{cpt}}I_{n-p}$ $(j=1,2)$ and $n_{i,j}'=p-n_{i,j} \in \mathbb{N}^{\oplus k}$. 
Next we suppose $p=n$.
Set $n_{i,j}':=n-n_{i,j}$. Then
\begin{eqnarray*}
\psi_n(\eta_1)^*x\psi_n(\eta_2)
&=&\sum_{i=1}^L \psi_n(\eta_1)^* \psi_{n_{i,1}}(\xi_{i,1}) \psi_{n_{i,2}} (\xi_{i,2})^* \psi_n(\eta_2)\\
&=&\sum_{i=1}^L \psi_{n_{i,1}'} (\xi_{i,1}')^* \psi_{n_{i,2}'}(\xi_{i,2}')
\end{eqnarray*}
where $\xi_{i,j}' \in X_{n_{i,j}'}^{\mathrm{cpt}}$ $(j=1,2)$. 
Hence for $0 \le p \le n$, we get
\begin{equation*}
\psi_p(\eta_1)^*x\psi_p(\eta_2)= \sum_{\{ i | n_{i,1},n_{i,2} \le p \} } \psi_{n_{i,1}'} (\xi_{i,1}')^* \psi_{n_{i,2}'}(\xi_{i,2}')
\end{equation*}
and 
\begin{equation*}
\psi_p(\eta_1)^*x_0\psi_p(\eta_2)= \sum_{\{ i | n_{i,1}=n_{i,2} \le p \} } \psi_p(\eta_1)^* \psi_{n_{i,1}}(\xi_{i,1}) \psi_{n_{i,2}} (\xi_{i,2})^* \psi_p(\eta_2)
=\psi_0(g_p).
\end{equation*}
For $Q \in C_0(\Lambda^M)$, by Lemma \ref{lem:unity} and Lemma \ref{lem:turn3}, there are $u_{i,j} \in X_{n_{i,j}'}^{\mathrm{cpt}}$ such that
\begin{align*}
\lefteqn{\sum_{i \in \mathcal{I}_0^c} \varphi^\psi_{M}(Q)\psi_{n_{i,1}'} (\xi_{i,1}')^* \psi_{n_{i,2}'}(\xi_{i,2}')\varphi^\psi_{M}(Q)}\hspace{1cm}\\
&=&\sum_{i \in \mathcal{I}_0^c} \psi_{n_{i,1}'} (\xi_{i,1}')^* \varphi^\psi_{n_{i,1}'+M}(u_{i,1} \widehat{\otimes} Q)
\varphi^\psi_{n_{i,2}'+M}(u_{i,2} \widehat{\otimes} Q)\psi_{n_{i,2}'}(\xi_{i,2}') \tag{$\sharp$}
\end{align*}
where $\mathcal{I}_0^c=\{ 1 \le i \le L \ | \ i \notin \mathcal{I}_0 \}$. 
We want to find $Q \in C_c(\Lambda^M)$ such that $(\sharp )=0$. 
Since $\psi$ is Nica covariant, for each $i \in \mathcal{I}_0^c$ we have
\begin{eqnarray*}
\lefteqn{
\varphi^\psi_{n_{i,1}'+M}(u_{i,1} \widehat{\otimes} Q)
\varphi^\psi_{n_{i,2}'+M}(u_{i,2} \widehat{\otimes} Q)}\hspace{0.5cm}\\
&=&\psi^{(n_{i,1 }' \vee n_{i,2}' +M)} \Bigl( 
\iota_{n_{i,1}'+M}^{(n_{i,1 }' \vee n_{i,2}') +M} \bigl( \pi_{n_{i,1}'+M}(u_{i,1} \widehat{\otimes} Q) \bigr)\\
& &\qquad \iota_{n_{i,2}'+M}^{(n_{i,1 }' \vee n_{i,2}') +M} \bigl(\pi_{n_{i,2}'+M}(u_{i,2} \widehat{\otimes} Q) \bigr)
\Bigr).
\end{eqnarray*}
where we used $(n_{i,1}'+M) \vee (n_{i,2}'+M)= (n_{i,1 }' \vee n_{i,2}') +M$. 
For proving $(\sharp )=0$, we enough to construct a compact neighborhood $U$ such that
\begin{equation*}
\mathrm{supp}(u_{i,1})U \vee \mathrm{supp}(u_{i,2})U = \emptyset
\end{equation*}
for all $i \in \mathcal{I}_0^c$ and a function $Q \in C_c(U)$. 
Since we suppose $\Lambda$ satisfies aperiodic condition, there is $\alpha \in v_0\Lambda^{\le \infty}$ such that $\alpha$ holds (A). 
Fix $i \in \mathcal{I}_0^c$.
\\
\textbf{Case 1, there is $1 \le l \le k$ such that $d(\alpha )_{(l)}<\infty$ and $(n_{i,1}')_{(l)} \neq (n_{i,2}')_{(l)}$.}\\
Set $n_i'=n_{i,1}' \vee n_{i,2}'$. 
We may suppose $d(\alpha ) \in \mathbb{N}^{\oplus k}$ with $\alpha \Lambda^{e_l}=\emptyset$ and $(n_i' -n_{i,2}')_{(l)} > 0$ for the above $l$. 
Set 
\begin{equation*}
K_{i,1}=\mathrm{supp}(u_{i,1}) ,\ K_{i,2}=\mathrm{supp}(u_{i,2})
\end{equation*}
If there is a non-empty compact neighborhood $U_{i,1}$ of $\alpha $ such that
$K_{i,1}U_{i,1} \vee K_{i,2}U_{i,1}=\emptyset$, 
then we finished. So we suppose $K_{i,1}U_{i,1} \vee K_{i,2}U_{i,1}\neq \emptyset$. 
Then
\begin{equation*}
\mathrm{Seg}_{(n_{i,2}', n_i' +d(\alpha ))}^{n_i' + d(\alpha )} ( K_{i,1}U_{i,1} \vee K_{i,2}U_{i,1} )
\end{equation*}
is a compact neighborhood because $\mathrm{Seg}_{n_{i,2}', n_i' +d(\alpha )}^{n_i' + d(\alpha )}$ is open continuous map. Since $(n_i'-n_{i,2}')_{(l)} > 0$ and $d(\alpha )_{(l)}<\infty$, we have
\begin{equation*}
\mathrm{Seg}_{(n_{i,2}', n_i' +d(\alpha ))}^{n_i' + d(\alpha )} ( K_{i,1}U_{i,1} \vee K_{i,2}U_{i,1} ) \vee \{ \alpha \} = \emptyset .
\end{equation*}
By Lemma \ref{lem:R-S-Y}, there is a compact neighborhood $U_{i,2}$ of $\alpha$ such that 
\begin{equation*}
\mathrm{Seg}_{(n_{i,2}', n_i' +d(\alpha ))}^{n_i' + d(\alpha )} ( K_{i,1}U_{i,1} \vee K_{i,2}U_{i,1} ) \vee U_{i,2} = \emptyset .
\end{equation*}
Set $U_i=U_{i,1} \cap U_{i,2}$. 
Suppose we can take an element $\lambda \in K_{i,1}U_i \vee K_{i,2}U_i \subset \Lambda^{n_i' + d(\alpha )}$. Then we have
\begin{equation*}
\lambda (n_{i,2}', n_i' +d(\alpha ) ) \in \mathrm{Seg}_{(n_{i,2}', n_i' +d(\alpha ))}^{n_i' + d(\alpha )} ( K_{i,1}U_{i,1} \vee K_{i,2}U_{i,1} )
\end{equation*}
and $\lambda (n_{i,2}', n_{i,2}'+d(\alpha )) \in U_{i,2}$. But there is no element with this property.

In this case, we set $M_i=d(\alpha )$.\\
\textbf{Case 2, for all $1 \le l \le k$, we have $d(\alpha )_{(l)}=\infty$ or if $d(\alpha )_{(l)}< \infty$, then $(n_{i,1}')_{(l)} = (n_{i,2}')_{(l)}$.}
\\
Set $n_{i,j}''=n_{i,j}' \wedge d(\alpha )$ for $j=1,2$ and $n_i''=n_{i,1}'' \vee n_{i,2}''$. 
We remark $n_{i,1}'' \neq n_{i,2}''$. 
By condition (A), there is $M_i \le d(\sigma^{n_{i,1}''}\alpha ) \wedge d(\sigma^{n_{i,2}''}\alpha )$ such that 
\begin{equation*}
\alpha (n_{i,1}'' , n_{i,1}''  +M_i) \neq \alpha (n_{i,2}'' ,n_{i,2}''+M_i).
\end{equation*} 
For $j=1,2$, take a neighborhood $U_{i,j}$ of $\alpha (n_i''+M_i)$ such that $U_{i,1} \cap U_{i,2} = \emptyset$. 
Then we remark that $d(\alpha ) \ge n_i''+M_i$. 
Define a neighborhood $U_i'$ of $\alpha (0,n_i'' + M_i )$ by
\begin{equation*}
U_i'= \bigcap_{j=1,2} \Bigl( \mathrm{Seg}_{(n_{i,j}'',n_{i,j}''+M_i)}^{n_{i}'' +M_i} \Bigr)^{-1}(U_{i,j}).
\end{equation*}
Take a compact neighborhood $U_i \subset U_i'$ of $\alpha (0, n_i' +M_i)$. 

Suppose we can take $\lambda \in K_{i,1}U_i \vee K_{i,2}U_i$. 
Then we have $\lambda (n_{i,1}', n_{i,1}'+n_i''+ M) \in U_i$ and 
$\lambda (n_{i,2}', n_{i,2}' +n_i''+ M) \in U_i$. Hence we obtain
\begin{equation*}
\mathrm{Seg}_{(n_{i,2}'',n_{i,2}''+M_i)}^{n_i'' +M_i}(\lambda (n_{i,1}',n_{i,1}'+n_i''+M_i))=\lambda (n_{i,1}'+n_{i,2}'', n_{i,1}'+n_{i,2}''+M_i) \in U_{i,2}
\end{equation*}
and 
\begin{equation*}
\mathrm{Seg}_{(n_{i,1}'',n_{i,1}''+M_i)}^{n_i'' +M_i}(\lambda (n_{i,2}',n_{i,2}'+n_i''+M_i))=\lambda (n_{i,1}''+n_{i,2}', n_{i.1}''+n_{i,2}'+M_i) \in U_{i,1}.
\end{equation*}
From the assumption in this case, we can easily check $n_{i,1}'+n_{i,2}''=n_{i,1}''+n_{i,2}'$. 
But this contradicts $U_{i,1} \cap U_{i,2}=\emptyset$. 

From the two subcases, for each $i \in \mathcal{I}_0$, we have a non-empty compact neighborhood $U_i \subset \Lambda^{M_i}$ of $\alpha (0,M_i)$ such that $\mathrm{supp}(u_{i,1})U_i \vee \mathrm{supp}(u_{i,2})U_i = \emptyset$. 
Define $M=\vee_{i \in \mathcal{I}_0^c} M_i$, $U=\vee_{i \in \mathcal{I}_0^c}U_i$. This $U$ satisfies $\mathrm{supp}(u_{i,1})U \vee \mathrm{supp}(u_{i,2})U = \emptyset$ for all $i \in \mathcal{I}_0^c$. 
Then $U$ is a non-empty compact neighborhood of $\alpha (0,M)$. Define $Q \in C_c(U)$ so that $0 \le Q \le 1$, $Q(\alpha (0,M))=1$.

If we define $b_j=\psi_p(\eta_j)\varphi_M^\psi (Q)$, then we obtain $b_1^*xb_2=b_1^*x_0b_2$ from $(\sharp)=0$. 
Moreover
\begin{eqnarray*}
\| x_0 \|
& < & |g_p(v_0)| + \epsilon
=|Q(\alpha (0,M) ) g_p(v_0) Q(\alpha (0,M) ) | + \epsilon \\
&\le& \| Q(g_p \circ r_M)Q \|_{C_0(\Lambda^M)} + \epsilon \\    
&=&\| \pi_M(Q (g_p \circ r_M) Q) \|_{\mathcal{L}(X_M)} + \epsilon \quad (\because \pi_M \mbox{ is injective})\\
&=&\|\psi^{(M)} (\pi_M(Q) \phi_M(g_p ) \pi_M(Q) ) \| + \epsilon \quad (\because \psi^{(M)} \mbox{ is injective since }\psi \mbox{ is injective})\\
&=&\| \varphi_M^\psi(Q)\psi_0(g_p) \varphi_M^\psi(Q) \| + \epsilon 
= \| \varphi_M^\psi(Q) \psi_p(\eta_1)^* x_0 \psi_p(\eta_2) \varphi_M^\psi(Q) \| + \epsilon \\
&=&\| b_1^*x b_2 \| + \epsilon
\end{eqnarray*}
Hence we complete the proof. 
\end{proof}

The following statement is the main theorem which is called Cuntz-Krieger type uniqueness theorem. 

\begin{thm}
\label{thm:CK-unique}
\normalfont\slshape
Let $\Lambda$ be a compactly aligned topological $k$-graph and $X$ be the product system arising from $\Lambda$. 
Suppose $\Lambda$ satisfies aperiodic condition and $\psi$ is an injective CNP-covariant representation. 
Then $\Pi \psi : \mathcal{NO}_X \longrightarrow C^*(\psi )$ is isometry.
\end{thm}

\begin{proof}
Take $x ,x_0\in C^*(\psi)^{\mathrm{cpt}}$ as in Proposition \ref{prop:freeness}. 
Since $\| x_0 \| \le \| x \|$ by Proposition \ref{prop:freeness} and $C^*(\psi )^{\mathrm{cpt}}$ is dense in $C^*(\psi )$, we can define $E_\psi : C^*(\psi ) \longrightarrow C^*(\psi )^{\mathrm{core}}$ such that $\Pi \psi \circ E = E_\psi \circ \Pi \psi$ on $\mathcal{NO}_X$. 
Suppose a positive element $x \in \mathcal{NO}_X$ satisfies $\Pi \psi (x)=0$. 
Then $\Pi \psi ( E (x) )=E_\psi (\Pi \psi (x))=0$ and by Corollary \ref{cor:core}, we get $E(x)=0$. 
Hence $x=0$ and this implies $\Pi \psi$ is injective.
\end{proof}

\begin{ex}
Consider $\Lambda_{\mathrm{BC}}$ defined in Example \ref{ex:BC-alg} and the associated product system $X$. 
Since $\Lambda_{\mathrm{BC}}$ satisfies aperiodic condition, Cuntz-Krieger uniqueness theorem holds for $\mathcal{NO}_X$. 
This gives another proof of \cite[Theorem 3.7]{Laca-Raeburn}. 
\end{ex}

\end{document}